\documentclass[reqno,11pt]{amsart}
\usepackage{amssymb}
\usepackage{latexsym,array}
\usepackage{amsfonts}
\newcommand{\C}{\mathbb {C}}

\newcommand{\cZ}{{\mathcal Z}}

\newcommand{\bl}{\boldsymbol{e_\ell}}
\newcommand{\br}{\boldsymbol{e_r}}
\newcommand{\bp}{\boldsymbol{\pi}}

\newcommand{\cX}{\mathcal {X}}

\newcommand{\bH}{\mathbb{H}}
\newcommand{\R}{\mathbb{R}}

\newtheorem{Pa}{Paper}[section]
\newtheorem{theorem}[Pa]{{\bf Theorem}}

\newtheorem{lemma}[Pa]{{\bf Lemma}}
\newtheorem{definition}[Pa]{{\bf Definition}}
\newtheorem{corollary}[Pa]{{\bf Corollary}}

\newtheorem{remark}[Pa]{{\bf Remark}}

\newtheorem{example}[Pa]{{\bf Example}}

\author[V. Bolotnikov]{Vladimir Bolotnikov}
\address{Williamsburg}
\title[Polynomial interpolation]{Polynomial interpolation  
over quaternions}
\begin{document}
\begin{abstract}
Interpolation theory for complex polynomials is well understood.
In the non-commutative quaternionic setting, the polynomials can
be evaluated ``on the left" and ``on the right". If the interpolation problem
involves interpolation conditions of the same (left or right) type, the results
are very much similar to the complex case: a consistent problem has a unique solution
of a low degree (less than the number of interpolation conditions imposed), and the
solution set of the homogeneous problem is an ideal in the ring
$\bH[z]$. The problem containing  both ``left" and ``right" interpolation conditions is
quite different: there may exist infinitely many low-degree solutions and the
solution set of the homogeneous problem is a quasi-ideal in $\bH[z]$.
\end{abstract}

\maketitle 
\section{Introduction}
\setcounter{equation}{0}

Given distinct points $z_1,\ldots,z_n\in\C$ and target values $c_1,\ldots,c_n\in\C$,
the Lagrange interpolation problem consists of finding a complex polynomial
$f\in\C[z]$ such that
\begin{equation}
f(z_i)=c_i\quad\mbox{for}\quad i=1,\ldots,n.
\label{1.1}
\end{equation}
A particular solution to this problem is the {\em Lagrange interpolation polynomial}
\begin{equation}
\widetilde{f}(z)=\sum_{j=1}^n \frac{c_j p_j(z)}{p_j(z_j)},\quad
\mbox{where}\quad p_j(z)=\prod_{\substack{i=1 \\ i\neq j}}^n (z-z_j),
\label{1.2}   
\end{equation}
while all polynomials satisfying conditions \eqref{1.1} are parametrized 
by the formula
\begin{equation}
f(z)=\widetilde{f}(z)+p(z)h(z), \quad p(z)=\prod_{i=1}^n(z-z_i), \quad h\in\C[z],
\label{1.3}   
\end{equation}
where $h$ is the free parameter. The latter representation holds since 
the mapping $f\mapsto (f(z_1),\ldots,f(z_n))$ is linear from $\C[z]$ to $\C^n$ and since 
the solution set of the corresponding homogeneous problem is the ideal in $\C[z]$ generated by $p$. 

\smallskip

Over the years, Lagrange interpolation has been played a prominent role in approximation theory and 
numerical analysis; more recent applications include image processing and control theory.
The problem can be settled exactly as in the complex case for polynomials over 
any field (including finite fields, which has applications in cryptography).
However, interpolation problems in non-commutative polynomial rings have not 
attracted much attention so far. The objective of this paper is to study the Lagrange interpolation 
problem for polynomials over 
the sqew field $\bH$ 
of real quaternions
\begin{equation}
\alpha=x_0+{\bf i}x_1+{\bf j}x_2+{\bf k}x_3 \qquad (x_0,x_1,x_2,x_3\in\mathbb R),
\label{1.4}
\end{equation}
where ${\bf i}, {\bf j}, {\bf k}$ are imaginary units commuting with the reals and such that
${\bf i}^2={\bf j}^2={\bf k}^2={\bf ijk}=-1$. We denote by $\bH[z]$ the ring of polynomials in one 
formal variable $z$ which commutes with quaternionic coefficients.
The ring operations in $\bH[z]$ are defined as in the commutative case, but as
multiplication in $\bH$ is not commutative, multiplication in $\bH[z]$ is not commutative
either. For $\alpha\in\bH$ and $f\in \bH[z]$, we define
$f^{\bl}(\alpha)$ and $f^{\br}(\alpha)$
(left and right evaluations of $f$ at $\alpha$) by
\begin{equation}
f^{\bl}(\alpha)=\sum_{k=0}^n\alpha^k f_k,\quad
f^{\br}(\alpha)=\sum_{k=0}^n f_k\alpha^k,\quad\mbox{if}\quad f(z)=\sum_{k=0}^n z^k f_k, \; 
\; f_k\in\bH.
\label{1.6}   
\end{equation}
Since $\R$ is the center of $\bH$, the ring $\R[z]$ of polynomials with
real coefficients is the center of $\bH[z]$. As a consequence of this observation,
we mention two cases where the left and the right evaluations produce the same result.
\begin{remark}
If $x\in\R$, then $f^{\bl}(x)=f^{\br}(x)$ for every $f\in\bH[z]$. On the other hand, if 
$f\in\R[z]$, then $f^{\bl}(\alpha)=f^{\br}(\alpha)$ for every $\alpha\in\bH$.
\label{R:1.1}
\end{remark}
In general, interpolation conditions imposed by the left and the right evaluations should be 
distinguished. We will consider the interpolation problem whose data set consists of 
two (not necessarily disjoint) finite sets 
\begin{equation}
\Lambda=\{\alpha_1,\ldots,\alpha_n\}\quad\mbox{and}\quad 
 \Omega=\{\beta_1,\ldots,\beta_m\}
\label{1.7}
\end{equation}
of distinct elements in $\bH$ along with the respective target values $c_1,\ldots,c_n$ and 
$d_1,\ldots,d_m$
in $\bH$. The {\em two-sided Lagrange problem} consists of finding an $f\in\bH[z]$ such that 
\begin{align}
f^{\bl}(\alpha_i)&=c_i\quad\mbox{for}\quad i=1,\ldots,n,\label{1.8}\\
f^{\br}(\beta_j)&=d_j\quad\mbox{for}\quad j=1,\ldots,m.\label{1.9}
\end{align}
The problem will be termed  {\em left-sided} if $ \Omega=\emptyset$ and  
{\em right-sided} if $\Lambda=\emptyset$.  
Since right and left evaluations coincide at real points, we may assign
all real interpolation nodes to the left set $\Lambda$ assuming therefore,
that $\Omega\cap \R=\emptyset$. We emphasize that the sets \eqref{1.7} do not have to be disjoint, so 
that we may  have left and right interpolation conditions at the same interpolation node 
$\alpha_i=\beta_j$.  One-sided interpolation problems in quaternionic and related non-commutative 
settings were previously 
discussed in \cite{lam1,hzw,eric,op}, mostly due to their connections with quaternionic 
Vandermonde matrices. These results are recalled in Section 3.1 below. 

\smallskip

The paper is organized as follows. Section 2 contains the background on quaternionic polynomials 
and their (left and right) zeros; left and right minimal polynomials are also recalled in Section 2.
In Section 3 we present certain necessary conditions for the problem
to have a solution and show that in case these conditions are met, one can assume without loss of 
generality that none three elements in the set $\Lambda\cup \Omega$ belong to the same conjugacy class.
It is shown in Section 4 that if some elements in $\Lambda$ have conjugates in $\Omega$, there are 
additional necessary (and this time, sufficient) conditions for the problem to be solvable given in 
terms of certain Sylvester equations and related to certain backward-shift operators on $\bH[z]$.
In Section 5 we present the parametrization of all solutions $f\in\bH[z]$ to the two-sided problem 
\eqref{1.8}, \eqref{1.9} in the form 
\begin{equation}
f(z)=\widetilde{f}(z)+f_h(z;\mu_1,\ldots\mu_k)+P_{\Lambda,{\boldsymbol\ell}}(z)\cdot h(z)\cdot 
P_{\Omega,{\bf r}}(z),
\label{1.11}
\end{equation}
where $\widetilde{f}$ is a particular low-degree solution, $f_h$ is the general low-degree
solution of the related homogeneous problem containing free parameters $\mu_1,\ldots,\mu_k$
(each parameter is associated to a pair $(\alpha_i,\beta_j)\in\Lambda\times\Omega$ of equivalent 
interpolation nodes and varies in a two-dimensional real subspace of $\bH$), 
$P_{\Lambda,{\boldsymbol\ell}}$ and $P_{\Omega,{\bf r}}$ are respectively the left and the right 
minimal polynomials of the sets $\Lambda$ and $\Omega$, and $h$ is the free parameter in $\bH[z]$.
The parametrization formula \eqref{1.11} somewhat resembles well known results on bi-tangential 
interpolation for matrix-valued complex polynomials \cite{bk,bcr}, in particular, the fact that 
in case of non-empty intersection of $\Lambda$ and $\Omega$, an additional condition is needed
to guarantee the uniqueness of a low-degree solution.

\section{Background}
\setcounter{equation}{0}

To start, let us fix notation and terminology. For $\alpha\in\bH$ of the form \eqref{1.4},
its real and imaginary parts, the quaternion conjugate and the absolute value
are defined as ${\rm Re}(\alpha)=x_0$, ${\rm Im}(\alpha)={\bf i}x_1+{\bf j}x_2+{\bf k}x_3$,
$\bar \alpha={\rm Re}(\alpha)-{\rm Im}(\alpha)$ and
$|\alpha|^2=\alpha\overline{\alpha}=|{\rm Re}(\alpha)|^2+|{\rm Im}(\alpha)|^2$,
respectively. As in the complex case, $\alpha +\overline{\alpha}=2{\rm 
Re}(\alpha)$.
\begin{definition}
{\rm Two quaternions $\alpha$ and $\beta$ are called {\em equivalent} (conjugate to each other)
if $\alpha=h^{-1}\beta h$ for some nonzero $h\in\mathbb H$.} 
\label{D:2.1}
\end{definition}
It follows (see e.g., \cite{fuzhen}) that
\begin{equation}
\alpha\sim\beta\quad\mbox{if and only if}\quad {\rm Re}(\alpha) ={\rm Re}(\beta) \;
\mbox{and} \; |\alpha|=|\beta|.
\label{2.1}
\end{equation}
Therefore, the conjugacy class of a given $\alpha\in\mathbb H$ form a $2$-sphere (of radius
$|{\rm Im}(\alpha)|$ around ${\rm Re}(\alpha)$) which will be denoted  by $[\alpha]$.
It is clear that $[\alpha]=\{\alpha\}$ if and only if $\alpha\in\mathbb R$.
\subsection{Polynomial conjugation} The conjugate of a polynomial $f$ is defined by
\begin{equation}
f^\sharp(z)=\sum_{j=0}^n z^j \overline{f}_j\quad\mbox{if}\quad f(z)=\sum_{j=0}^n z^j f_j.
\label{2.2}
\end{equation}
The anti-linear involution $f\mapsto f^\sharp$ can be viewed as an extension of the   
quaternionic conjugation $\alpha\mapsto \overline{\alpha}$ from $\bH$ to $\bH[z]$.
A polynomial $f$ is real if and only if $f\equiv f^\sharp$. Some further properties of
polynomial conjugation are listed below:
\begin{equation}
ff^\sharp=f^\sharp f,\quad (fg)^\sharp=g^\sharp f^\sharp,\quad (fg)(fg)^\sharp=f(gg^\sharp)
f^\sharp=(ff^\sharp)(gg^\sharp).
\label{2.3}
\end{equation}

\subsection{Left and right zeros of polynomials}
We now recall some results on quaternionic polynomials and their roots
needed for the subsequent analysis and presented here in the form suitable
for our needs. For a more detailed exposition, we refer to \cite{lam}
and references therein.
\begin{definition}
{\rm An element $\alpha\in\bH$ is a {\em left zero} of $f\in\bH[z]$ if 
$f^{\bl}(\alpha)=0$, and it is a {\em right zero} of $f$ if $f^{\br}(\alpha)=0$.}
\label{D:2.2}   
\end{definition}
We will denote by $\cZ_{\boldsymbol\ell}(f)$ and $\cZ_{{\boldsymbol r}}(f)$ the sets of all
left and all right zeros of $f$ respectively.  It follows from Remark \ref{R:1.1} that 
for a real polynomial $f\in\R[z]$, these two sets coincide: $\cZ(f):=\cZ_{\boldsymbol\ell}(f)=\cZ_{\bf 
r}(f)$.
\begin{example}
{\rm For a non-real quaternion $\alpha$, the real polynomial
\begin{equation}
\cX_{[\alpha]}(z)=(z-\alpha)(z-\overline{\alpha})=(z-\overline{\alpha})(z-\alpha)=
z^2-2z\cdot {\rm Re}(\alpha)+|\alpha|^2
\label{2.4}   
\end{equation}
is called the {\em characteristic polynomial} of the conjugacy class $[\alpha]$
(by \eqref{2.1}, $\cX_{[\alpha]}=\cX_{[\beta]}$ if and only if $\alpha\sim\beta$) and can be
characterized as a unique monic quadratic polynomial with the zero set equal $[\alpha]$.
It is irreducible over $\R$ since ${\rm Re}(\alpha)<|a|$ and it is easily verified that  
conversely, any monic real quadratic polynomial without real roots is the characteristic
polynomial
of a unique quaternionic conjugacy class.}
\label{E:2.3}
\end{example}
If $f\in\R[z]$, then for each $\alpha\in\bH$ and $h\neq 0$, 
we have $f(h^{-1}\alpha h)=h^{-1}f(\alpha)h$ so that
$\cZ(f)$ contains, along with each $\alpha$, the whole conjugacy class $[\alpha]$.
For any $f\in\bH[z]$, the polynomial $ff^\sharp$ is real and therefore, $\cZ(ff^\sharp)$ is the 
union of finitely many conjugacy classes. The following result is due I. Niven \cite{niv}.
\begin{theorem}
Let $\deg (f)\ge 1$ and let $\cZ(ff^\sharp)=\bigcup V_i$ be 
the union of distinct conjugacy classes.
Then $\cZ_{\boldsymbol\ell}(f)\bigcup \cZ_{\bf r}(f)\subset\cZ(ff^\sharp)$ and each 
conjugacy class contains at least one left and at least one right zero of $f$.
\label{T:2.4}
\end{theorem}
Since any real polynomial of positive degree has a complex root, the latter theorem implies that 
for any $f\in\bH[z]$ of positive degree, the zero sets $\cZ_{\boldsymbol\ell}(f)$ and 
$\cZ_{\bf r}(f)$ are not empty (the Fundamental Theorem of Algebra in $\bH[z]$).
\begin{remark}  
As a consequence of the Euclidean algorithm which holds in $\bH[z]$ in both left and right versions, we 
have
\begin{align}
&\alpha\in\cZ_{\boldsymbol\ell}(f) \; \; \Longleftrightarrow \; f(z)=(z-\alpha)h(z)\quad
\mbox{for some}\quad h\in\bH[z];\label{2.5}\\   
&\alpha\in\cZ_{\boldsymbol r}(f) \; \; \Longleftrightarrow \; f(z)=\widetilde{h}(z)(z-\alpha)\quad
\mbox{for some}\quad \widetilde{h}\in \bH[z].\notag
\end{align}
\label{R:2.5}
\end{remark}
It follows from \eqref{2.5}, that if $g^{\bl}(\alpha)=0$, then $(gf)^{\bl}(\alpha)=0$ 
for any $f\in\bH[z]$.
On the other hand, since
$(gf)(z)=\sum_{k=0}^n z^k g(z)f_k$, we also have
$$ 
(gf)^{\bl}(\alpha)=g^{\bl}(\alpha)\sum_{k=0}^n (g^{\bl}(\alpha)^{-1}\alpha
g^{\bl}(\alpha))^kf_k=g^{\bl}(\alpha)f^{\bl}(g^{\bl}(\alpha)^{-1}\alpha g^{\bl}(\alpha)),
$$
provided  $g^{\bl}(\alpha)\neq 0$. Therefore, the left evaluation of the product
of two polynomials is defined by the formula
\begin{equation}
(gf)^{\bl}(\alpha)=\left\{\begin{array}{ccc} 
g^{\bl}(\alpha)\cdot f^{\bl}\left(g^{\bl}(\alpha)^{-1}\alpha 
g^{\bl}(\alpha)\right)&\mbox{if} &
g^{\bl}(\alpha)\neq 0, \\
0 & \mbox{if} & g^{\bl}(\alpha)= 0.\end{array}\right.
\label{2.6}   
\end{equation}
Similarly, the right evaluation of the product is given by 
\begin{equation}
(gf)^{\br}(\alpha)=\left\{\begin{array}{ccc} g^{\br}\left(f^{\br}(\alpha)\alpha 
f^{\br}(\alpha)^{-1}\right)\cdot
f^{\br}(\alpha)&\mbox{if} & f^{\br}(\alpha)\neq 0, \\
0 & \mbox{if} & f^{\br}(\alpha)= 0.\end{array}\right.
\label{2.7}   
\end{equation}
Note that in case $\alpha\in\R$, both \eqref{2.6} and \eqref{2.7} simplify to
$(gf)(\alpha)=g(\alpha)f(\alpha)$.
\begin{lemma}
Let $f\in\bH[z]$ and let $\alpha,\beta\in\bH$ be two distinct conjugates:
$\beta\in[\alpha]\backslash\{\alpha\}$. The following are equivalent:
\begin{enumerate}
\item[(1)] $\alpha,\beta\in Z_{\boldsymbol\ell}(f)$;\qquad $(2)$ $\alpha,\beta\in Z_{\boldsymbol 
r}(f)$;
\qquad $(3)$ $[\alpha]\subset Z_{\boldsymbol\ell}(f)\cap
Z_{\boldsymbol r}(f)$;
\item[(4)] $f$ can be factored as
\begin{equation}
f(z)=\cX_{[\alpha]}(z)g(z)=g(z)\cX_{[\alpha]}(z)\quad\mbox{for some}\quad g\in\bH[z].
\label{2.8}
\end{equation}
\end{enumerate}
\label{L:2.6}
\end{lemma}
\begin{proof}
Implications $(4)\Rightarrow(3)\Rightarrow(2)$ and $(3)\Rightarrow(1)$ are
trivial. To confirm  $(1)\Rightarrow(4)$, let us assume that $\alpha,\beta\in
Z_{\boldsymbol\ell}(f)$.  Since $\beta\in[\alpha]$, then
$$
\beta^2-\beta(\alpha+\overline{\alpha})+|\alpha|^2=\cX_{[\alpha]}(\beta)=0
$$
so that $\beta(\beta-\alpha)=\beta^2-\beta
\alpha=\beta\overline{\alpha}-|\alpha|^2=(\beta-\alpha)\overline{\alpha}$, and we conclude:
\begin{equation}
(\beta-\alpha)^{-1}\beta(\beta-\alpha)=\overline{\alpha}, \quad\mbox{whenever}\quad
\beta\in[\alpha]\backslash\{\alpha\}.
\label{2.9}
\end{equation}
Since $f^{\bl}(\alpha)=0$, $f$ can be factored as in \eqref{2.5}. We use 
the latter factorization to left-evaluate $f$ at $\beta$; according to \eqref{2.6} and 
\eqref{2.9},
\begin{equation}
f^{\bl}(\beta)=(\beta-\alpha)\cdot h^{\bl}((\beta-\alpha)^{-1}\beta 
(\beta-\alpha))=(\beta-\alpha)\cdot h^{\bl}(\overline{\alpha}).
\label{2.10}
\end{equation}
Since $f^{\bl}(\beta)=0$ and since $\bH$ is a division ring, we conclude from \eqref{2.10} that
$h^{\bl}(\overline{\alpha})=0$. Then again by \eqref{2.5}, $h$ can be factored as
$h(z)=(z-\overline{\alpha})\cdot g(z)$ for some $g\in\bH[z]$ which being combined with factorization 
\eqref{2.5} for $f$ gives \eqref{2.9}:
$f(z)=(z-\alpha)\cdot h(z)=(z-\alpha)(z-\overline{\alpha})\cdot g(z)=\cX_{[\alpha]}(z)\cdot g(z)$.
\end{proof}
The last lemma supplements Theorem \ref{T:2.4} as follows:
\begin{remark}
Each conjugacy class $V\subset\bH$ 
containing zeros of an  $f\in\bH[z]$ either contains exactly one left and exactly one 
right zero of $f$ or $V\subset Z_{\boldsymbol\ell}(f)\cap Z_{\bf r}(f)$. 
\label{R:2.6a}
\end{remark}
\subsection{Minimal polynomials} Since the division algorithm (left and right) holds in
$\bH[z]$, any ideal (left or right) is principal. Given a set $\Lambda\subset\bH$,
the sets
\begin{equation}
{\mathbb I}_{\Lambda,{\boldsymbol \ell}}:=\{f\in\bH[z]: \;  \cZ_{\boldsymbol \ell}\supset 
\Lambda\}\quad\mbox{and}\quad 
{\mathbb I}_{\Lambda,{\bf r}}:=\{f\in\bH[z]: \;  \cZ_{\bf r}\supset\Lambda\}
\label{2.11}
\end{equation}
are respectively, a right and a left ideal in $\bH[z]$, which are non-trivial if and only if $\Lambda$ is 
contained in a finite union of conjugacy classes.  In the latter case, ${\mathbb I}_{\Lambda,{\boldsymbol \ell}}$ 
and ${\mathbb I}_{\Lambda,{\bf r}}$) are generated
by (unique) monic polynomials $P_{\Lambda,{\boldsymbol \ell}}$ and $P_{\Lambda,{\bf r}}$ which will be called 
the {\em left} and the {\em right minimal polynomials} (abbreviated to {\bf lmp} and {\bf rmp}, respectively, in 
what follows) of $\Lambda$. They can be equivalently
defined as unique monic polynomials of the lowest degree with the left (respectively, right) zero set equal
$\Lambda$. Since ${\mathbb I}_{\emptyset,{\boldsymbol \ell}}={\mathbb I}_{\emptyset,{\bf 
r}}=\bH[z]$, it makes sense to define the {\bf lmp} and {\bf rmp} of the empty set by 
letting $P_{\emptyset,{\boldsymbol \ell}}=P_{\emptyset,{\bf r}}\equiv 1$.
The next observation is a consequence of Lemma \ref{L:2.6}.
\begin{remark}
Let $\Lambda\subset \bH$ be contained in a finite union of conjugacy classes.
If $V$ is a conjugacy class disjoint with $\Lambda$ and if $U\subset V$ contains at least
two elements, then  $P_{\Lambda\cup U,{\boldsymbol
\ell}}(z)=\cX_V(z)P_{\Lambda,{\boldsymbol \ell}}(z)$ and $P_{\Lambda\cup 
U,{\bf r}}(z)=\cX_V(z)P_{\Lambda,{\bf r}}(z)$.
\label{R:3.6}
\end{remark}
\begin{theorem}
Let $\Lambda\subset\bH$ be arranged as
$\Lambda=\{\alpha_1,\ldots,\alpha_s\}\cup U_1\cup\ldots \cup U_k,
$
where $U_1,\ldots,U_k$ are subsets of disjoint conjugacy classes
$V_1,\ldots,V_k$ respectively containing at least two elements each, and
$\alpha_1,\ldots,\alpha_s\in\bH\backslash
(V_1\bigcup\ldots\bigcup V_k)$ are non-equivalent quaternions. Then
\begin{equation}
P_{\Lambda,{\boldsymbol \ell}}(z)=p_s(z)\cdot \prod_{j=1}^k \cX_{V_j}(z)\quad\mbox{and}
\quad P_{\Lambda,{\bf r}}(z)=q_s(z)\cdot\prod_{j=1}^k \cX_{V_j}(z)
\label{2.12}
\end{equation}
where $p_s$ is the monic polynomial of degree $s$ obtained from the recursion
\begin{equation}
p_0(z)\equiv 1,\quad p_{j+1}(z)=p_{j}(z)\left(z-p^{\bl}_{j}(\alpha_{j+1})^{-1}\alpha_{j+1}
p^{\bl}_{j}(\alpha_{j+1})\right),\label{2.14}
\end{equation} 
and $q_s$ is the monic polynomial of degree $s$ obtained from the recursion
\begin{equation}
q_0(z)\equiv 1,\quad q_{j+1}(z)=(z-q^{\br}_{j}(\alpha_{j+1})\alpha_{j+1}
q^{\br}_{j}(\alpha_{j+1})^{-1})\cdot q_{j}(z).
\label{2.13}
\end{equation}
In particular, $\deg (P_{\Lambda,{\boldsymbol\ell}})=\deg (P_{\Lambda,{\bf r}})=s+2k$.
\label{T:3.7}   
\end{theorem}
\begin{proof} 
Based on the fact that all $\alpha_j$'s belong to distinct conjugacy classes, the 
induction argument shows that for each $j=1,\ldots,s$, the polynomial $p_j(z)$ is the {\bf
lmp} of the set $\{\alpha_1,\ldots,\alpha_j\}$ and that $p^{\bl}_{j}(\alpha_{j+1})\neq 0$ (so that the 
recursion 
formula \eqref{2.14} makes sence). Thus, $p_s$ is the {\bf lmp} of the set $\{\alpha_1,\ldots,\alpha_s\}$, 
and  the repeated application of Remark \ref{R:3.6} leads us to
$$
P_{\Lambda,{\boldsymbol\ell}}(z)=
P_{\{\alpha_1,\ldots,\alpha_s\},{\boldsymbol\ell}}(z)\cdot \prod_{j=1}^k \cX_{V_j}(z)=p_s(z)\cdot \prod_{j=1}^k 
\cX_{V_j}(z),
$$
i.e., to the first formula in \eqref{2.12}. The second formula follows in much the same way. The final 
statement of  the theorem is an obvious consequence of \eqref{2.14}, \eqref{2.13}.
\end{proof}
Recursion \eqref{2.14} was carried out under the assumption that all $\alpha_j$'s are non-equivalent.
It is worth noting that a slight modification of \eqref{2.14} applies to {\em any} finite set
as follows:  {\em the {\bf lmp} of the set $\Lambda=\{\alpha_1,\ldots,\alpha_n\}\subset \bH$
equals the polynomial $p_n(z)$ obtained recursively by letting $p_0(z)\equiv 1$ and
\begin{equation}
p_{j+1}(z)=\left\{\begin{array}{ccc} p_{j}(z)\left(z-p^{\bl}_{j}(\alpha_{j+1})^{-1}\alpha_{j+1}
p^{\bl}_{j}(\alpha_{j+1})\right)&\mbox{if}& p^{\bl}_{j}(\alpha_{j+1})\neq 0,\\
p_{j}(z)&\mbox{if}& p^{\bl}_{j}(\alpha_{j+1})= 0.\end{array}\right.
\label{2.15}
\end{equation}}
We omit the straightforward inductive proof and the formulation of the  right-sided version of 
\eqref{2.15}.
Instead, we include several final remarks.

\smallskip

It can be shown (again by induction) that $p^{\bl}_{j}(\alpha_{j+1})=0$ if only if
at least two elements in $\{\alpha_1,\ldots\alpha_j\}$ are conjugates of $\alpha_{j+1}$.
Thus, the recursion \eqref{2.15} takes into account
the two first elements from the same conjugacy class $V$ and dismisses all further elements from
$V$. In particular, the original recursion \eqref{2.14} applies to any finite set $\Lambda\subset \bH$,
none three elements of which are equivalent.

\smallskip

If $V$ is a conjugacy class and if
$V\cap\Lambda=\{\alpha_{i_1},\alpha_{i_2},\ldots\}$, then 
$\cZ_{\boldsymbol \ell}(p_j)\cap V=\emptyset$ for $j<i_1$ and
$\cZ_{\boldsymbol \ell}(p_j)\cap V=\{\alpha_{i_1}\}$ for $i_1\le j<i_2$.
Since $\alpha_{i_2}\neq \alpha_{i_1}$, we have $p^{\bl}_{i_2-1}(\alpha_{i_2})\neq 0$.
The polynomial $p_{i_2}$ defined by the top formula in \eqref{2.15} has two
distinct left zeros $\alpha_{i_1}, \, \alpha_{i_2}$ in $V$ and therefore,
$V\subset \cZ_{\boldsymbol \ell}(p_j)\bigcap \cZ_{\bf r}(p_j)$ for all $j\ge i_2$,
by Lemma \ref{L:2.6} and formulas \eqref{2.6}, \eqref{2.7}. 

\smallskip

Recursion \eqref{2.15} produces the minimal polynomial $P_{\Lambda,{\boldsymbol \ell}}$
as a product of linear factors. Although the outcome $P_{\Lambda,{\boldsymbol \ell}}$
does not depend on the order in which the elements of $\Lambda$ are arranged,
different permutations of $\Lambda$ produce via recursion \eqref{2.15} different factorizations of 
$P_{\Lambda,{\boldsymbol \ell}}$. Factorization \eqref{2.12} comes up if $\Lambda$ is arranged 
so that the two first appearances of the elements from each conjugacy class occur consecutively,
that is, if $\alpha_j$ is not equivalent to $\alpha_1,\ldots,\alpha_{j-1}$, then either 
$\alpha_{j}\sim\alpha_{j+1}$ or $\Lambda\backslash\{\alpha_j\}$ contains no element conjugate to $\alpha_j$.

\section{Consistency of interpolation conditions and simple cases}
\setcounter{equation}{0}

In the complex setting, the Lagrange problem
with distinct interpolation  nodes is always consistent. In the quaternionic case,
inconsistency may occur if the set
$\Lambda\cup\Omega$ contains more than two points
from the same conjugacy class; the one-sided version of this phenomenon was observed
in \cite{genstr} in a more general setting of slice regular functions.
\begin{lemma}
For $f\in\bH[z]$ and three distinct equivalent quaternions $\alpha,\beta,\gamma$,
\begin{align}
f^{\bl}(\gamma)=&(\gamma-\beta)(\alpha-\beta)^{-1}f^{\bl}(\alpha)+(\alpha-\gamma)
(\alpha-\beta)^{-1}f^{\bl}(\beta),\label{3.1}\\
f^{\br}(\gamma)=&(\alpha-\beta)^{-1}f^{\bl}(\alpha)\gamma-
\beta(\alpha-\beta)^{-1}f^{\bl}(\alpha)\notag\\
\quad 
&+\alpha(\alpha-\beta)^{-1}f^{\bl}(\beta)-(\alpha-\beta)^{-1}f^{\bl}(\beta)\gamma,\label{3.2}\\
f^{\br}(\gamma)=&f^{\br}(\alpha)(\alpha-\beta)^{-1}(\gamma-\beta)+f^{\br}(\beta)(\alpha-\beta)^{-1}
(\alpha-\gamma),\label{3.3}\\
f^{\bl}(\gamma)=&\gamma 
f^{\br}(\alpha)(\alpha-\beta)^{-1}-f^{\br}(\alpha)(\alpha-\beta)^{-1}
\beta\notag\\
&+f^{\br}(\beta)(\alpha-\beta)^{-1}\alpha-\gamma f^{\br}(\beta)(\alpha-\beta)^{-1}.\label{3.4}
\end{align}
\label{L:4.1}
\end{lemma}
\begin{proof} The polynomial 
$f(z)-f^{\bl}(\alpha)-(z-\alpha)(\alpha-\beta)^{-1}(f^{\bl}(\alpha)-f^{\bl}(\beta))$
has left zeros at $\alpha$ and $\beta$ and hence, by Lemma \ref{L:2.6}
\begin{equation}
f(z)=f^{\bl}(\alpha)+(z-\alpha)(\alpha-\beta)^{-1}(f^{\bl}(\alpha)-f^{\bl}(\beta))+
\cX_{[\alpha]}(z)g(z)
\label{3.5}  
\end{equation}
for some $g\in\bH[z]$. Since $\gamma\in[\alpha]$, it holds that 
$(\cX_{[\alpha]}g)^{\bl}(\gamma)=
(\cX_{[\alpha]}g)^{\br}(\gamma)=0$ and then left and right evaluations of \eqref{3.5} at $c$
give 
\begin{align*}
f^{\bl}(\gamma)=&f^{\bl}(\alpha)+(\gamma-\alpha)(\alpha-\beta)^{-1}(f^{\bl}(\alpha)-f^{\bl}(\beta)),\\
f^{\br}(\gamma)=&f^{\bl}(\alpha)+(\alpha-\beta)^{-1}(f^{\bl}(\alpha)-f^{\bl}(\beta))\gamma\\
&-\alpha(\alpha-\beta)^{-1}(f^{\bl}(\alpha)-f^{\bl}(\beta)),
\end{align*}
which are equivalent to \eqref{3.1} and \eqref{3.2}, respectively. Relations \eqref{3.3}
and \eqref{3.4} are established similarly, by applying Lemma \ref{L:2.6}
to the polynomial 
$f(z)-f^{\br}(\alpha)-(f^{\br}(\alpha)-f^{\br}(\beta))(\alpha-\beta)^{-1}(z-\alpha)$.
\end{proof}
Lemma \ref{L:4.1} shows that left (or right) evaluations of $f\in\bH[z]$ at any two points
from the same conjugacy class uniquely determine left {\em and} right evaluations of $f$ at any point
in this conjugacy class. Thus, if the set  $\Lambda\bigcup \Omega$ contains more than two points from the same 
conjugacy class, the corresponding target values must satisfy certain conditions (outlined in Lemma \ref{L:4.1}) 
for the Lagrange problem to have a solution.

\smallskip

Let $V$ be a conjugacy class such that $V\bigcap
\Lambda=\{\alpha_{i_1},\alpha_{i_2},\ldots,\alpha_{i_k}\}$ contains at least 
 two elements and let $V\bigcap\Omega=\{\beta_{j_1},\ldots,\beta_{j_s}\}$.
For the assigned target values $c_{i_\ell}$  and $d_{j_r}$, we verify equalities (cf.
\eqref{3.1} and \eqref{3.2})
\begin{align}
c_{i_\ell}&=(\alpha_{i_\ell}-\alpha_{i_2})(\alpha_{i_1}-\alpha_{i_2})^{-1}c_{i_1}+
(\alpha_{i_1}-\alpha_{i_\ell})(\alpha_{i_1}-\alpha_{i_2})^{-1}c_{i_2},\label{3.6}\\
d_{j_r}&=(\alpha_{i_1}-\alpha_{i_2})^{-1}c_{i_1}\beta_{j_r}-
\alpha_{i_2}(\alpha_{i_1}-\alpha_{i_2})^{-1}c_{i_1}\notag\\
&\qquad+\alpha_{i_1}(\alpha_{i_1}-\alpha_{i_2})^{-1}c_{i_2}-(\alpha_{i_1}-\alpha_{i_2})^{-1}c_{i_2}
\beta_{j_r}\notag
\end{align}  
for $\ell=3,\ldots,k$ and $r=1,\ldots,s$. If at least one of them fails, the
Lagrange problem \eqref{1.8}, \eqref{1.9} does not have solutions, by Lemma \ref{L:4.1}.
Otherwise,  any polynomial $f\in\bH[z]$ satisfying
interpolation conditions \eqref{1.8} at $\alpha_{i_1}$ and $\alpha_{i_2}$ will satisfy
left interpolation conditions at $\alpha_{i_\ell}$ (for $\ell=3,\ldots,k$) and
right interpolation conditions at $\beta_{j_r}$ (for $r=1,\ldots,s$) automatically, again
by Lemma \ref{L:4.1}. Hence, removing interpolation conditions at these points
we get a reduced interpolation problem with the same solution set as the original one.
Alternatively, if $V\bigcap\Omega$ contains at least two elements $\beta_{j_1},\beta_{j_2}$,
we may use relations \eqref{3.3} and \eqref{3.4} to check if other interpolation conditions on
$V$ are compatible with those two at $\beta_{j_1}$ and $\beta_{j_2}$ and, if this is the case,
all other conditions can be removed without affecting the solution set of the problem.
After completing consistency verifications in all conjugacy classes having more than two
common elements with $\Lambda\bigcup\Omega$, we either conclude that the original problem is
inconsistent or reduce it to a problem for which 
$$
{\bf (A)} \quad \mbox{none three of the interpolation nodes are equivalent.}\qquad\qquad
$$
The latter assumption is all we need to handle one-sided interpolation problems.
\subsection{One-sided interpolation}
In case $\Omega=\emptyset$, all consistency equalities are of the form \eqref{3.6}.
Let us assume that all of them hold true and start with the left Lagrange problem 
satisfying the 
assumption ({\bf A}). 

\smallskip

Let $P_{\Lambda,{\boldsymbol\ell}}$ be the {\bf lmp} of $\Lambda$ and let
$P_{{\Lambda}_k,{\boldsymbol\ell}}$ be
the {\bf lmp}  of the set ${\Lambda}_k:=\Lambda\backslash\{\alpha_k\}$ for
$k=1,\ldots,n$. It follows from Theorem \ref{T:3.7} that under assumption ${\bf (A)}$,
\begin{equation}
\deg (P_{\Lambda,{\boldsymbol\ell}})=n\quad \mbox{and}\quad \deg
(P_{{\Lambda}_k,{\boldsymbol\ell}})=n-1 \quad (k=1,\ldots,n).
\label{3.7}   
\end{equation}
\begin{theorem} 
Assume that none three elements in the set $\Lambda=\{\alpha_1,\ldots,\alpha_n\}$ are equivalent.
All polynomials $f\in\bH[z]$ satisfying left interpolation conditions \eqref{1.8} are given by the
formula
\begin{equation}
f=\widetilde{f}_{\boldsymbol\ell}+P_{\Lambda,{\boldsymbol\ell}}h,\quad
\widetilde{f}_{\boldsymbol\ell}(z)=\sum_{k=1}^n P_{{\Lambda}_k,{\boldsymbol\ell}}(z)\cdot
P^{\bl}_{{\Lambda}_k,{\boldsymbol\ell}}(\alpha_k)^{-1}\cdot c_k,\quad h\in\bH[z]
\label{3.8}
\end{equation}
with free parameter $h\in\bH[z]$. The left Lagrange polynomial $\widetilde{f}_{\boldsymbol\ell}$ is a unique 
solution to the problem \eqref{1.8} of degree less than $n$.
\label{T:5.1b}
\end{theorem} 
\begin{proof}
The polynomial $P_{{\Lambda}_k,{\boldsymbol\ell}}$ left-vanishes on $\Lambda_k$, i.e., 
$P^{\bl}_{{\Lambda}_k,{\boldsymbol\ell}}(\alpha_i)=0 $ for $i\neq k$. On the other hand,
$P^{\bl}_{{\Lambda}_k,{\boldsymbol\ell}}(\alpha_k)\neq 0$, since otherwise, the
set $\widetilde{\Lambda}_k$ contains at least two elements equivalent to $\alpha_k$ which contradicts the
assumption ${\bf (A)}$. Now it is easily verified that the polynomial $\widetilde{f}_{\boldsymbol\ell}$ 
defined as in \eqref{3.8} satisfies conditions \eqref{1.8}. Due to \eqref{3.7}, we have 
$\deg (\widetilde{f}_{\boldsymbol\ell})<n$. A polynomial $f$  satisfies conditions \eqref{1.8}
if and only if $\cZ_{\boldsymbol\ell}(f-\widetilde{f}_{\boldsymbol\ell})\supset\Lambda$, i.e., 
if and only if $f-\widetilde{f}_{\boldsymbol\ell}$ belongs to the right ideal ${\mathbb 
I}_{\Lambda,\boldsymbol \ell}$ \eqref{2.11}, which means that 
$f=\widetilde{f}_{\boldsymbol\ell}+P_{\Lambda,{\boldsymbol\ell}}h$ for some $h\in\bH[z]$.
Finally, if $h\not\equiv 0$, then $\deg(f)\ge \deg (P_{\Lambda,{\boldsymbol\ell}})=n>
\deg(\widetilde{f}_{\boldsymbol\ell})$.
\end{proof}
The right-sided problem is handled in much the same way. Let $P^{\br}_{\Omega,{\bf r}}$ be the {\bf rmp} of 
$\Omega=\{\beta_1,\ldots,\beta_m\}$ and let $P^{\br}_{{\Omega}_k,{\bf r}}$ denote the 
{\bf rmp} of the set ${\Omega}_j:=\Omega\backslash\{\beta_j\}$ for $j=1,\ldots,m$.
Under assumption ${\bf (A)}$, we have 
\begin{equation}
\deg (P_{\Omega,{\bf r}})=m\quad \mbox{and}\quad \deg (P_{{\Omega}_j,{\bf r}})=m-1 \quad 
(j=1,\ldots,m).
\label{3.9}  
\end{equation}
\begin{theorem}
Assume that none three elements in the set $\Omega=\{\beta_1,\ldots,\beta_m\}$ are equivalent.
Then all $f\in\bH[z]$ satisfying right conditions \eqref{1.9} are given by the
formula
\begin{equation}
f=\widetilde{f}_{\bf r}+hP_{\Lambda,{\bf r}},\quad \widetilde{f}_{\bf r}(z)=\sum_{k=1}^m d_k\cdot 
P^{\br}_{{\Omega}_k,{\bf r}}(\beta_k)^{-1}\cdot
P_{{\Omega}_k,{\bf r}}(z),\quad h\in\bH[z]
\label{3.10}
\end{equation}
with free parameter $h\in\bH[z]$. The right Lagrange polynomial $\widetilde{f}_{\bf r}$ is a unique
solution to the problem \eqref{1.9} of degree less than $m$.
\label{T:5.1a}  
\end{theorem}

\section{Backward shift operators and Sylvester equations}
\setcounter{equation}{0}
With any $\alpha\in\bH$, we can associate linear operators
$L_\alpha$ and $R_\alpha$
(left and right backward shifts) acting on $\bH[z]$ (which is now considered as
a linear vector space over $\bH$):
\begin{align}
L_\alpha: \; f(z)=\sum_{k=0}^n z^k f_k\mapsto \sum_{k=0}^{n-1}\left(\sum_{j=0}^{n-k-1}
\alpha^jf_{k+j+1}\right)z^k,\label{4.1}\\
R_\alpha: \; f(z)=\sum_{k=0}^n z^k f_k\mapsto \sum_{k=0}^{n-1}\left(\sum_{j=0}^{n-k-1}
f_{k+j+1}\alpha^j\right)z^k.\label{4.1a}
\end{align}
The terminology is partly justified by the fact that $L_\alpha f$ and $R_\alpha f$ are
the unique polynomials such that
\begin{equation}
f(z)=f^{\bl}(\alpha)+(z-\alpha)\cdot(L_\alpha f)(z)=f^{\br}(\alpha)+(R_\alpha f)(z)\cdot(z-\alpha).
\label{4.2}
\end{equation}
It follows directly from \eqref{4.1} that for any $\alpha,\beta\in\bH$,
\begin{equation}
(L_\alpha f)^{\br}(\beta)=(R_\beta f)^{\bl}(\alpha)=
\sum_{k=0}^{n-1}\sum_{j=0}^{n-k-1} \alpha^kf_{k+j+1}\beta^{k-j}.
\label{4.3}
\end{equation}
\begin{remark}
For $f\in\bH[z]$ and $\alpha, \beta\in\bH$,
\begin{equation}
\alpha \cdot(L_\alpha f)^{\br}(\beta)-(L_\alpha f)^{\br}(\beta)\cdot\beta=f^{\bl}(\alpha)-f^{\br}(\beta).
\label{4.4}
\end{equation}
\label{R:4.4a}
\end{remark} 
Indeed, the right evaluation at $z=\beta$ applied to the first representation in \eqref{4.2} 
gives
$f^{\br}(\beta)=f^{\bl}(\alpha)+(L_\alpha f)^{\br}(\beta)\cdot\beta-\alpha \cdot(L_\alpha f)^{\br}(\beta)$
which is equivalent to \eqref{4.4}. Our next objective is to determine to what extent the value
of $(L_\alpha f)^{\br}(\beta)$ can be recovered from the equality \eqref{4.4}. The next 
result is known.
\begin{lemma}
Given two non-equivalent $\alpha,\beta\in\bH$, the Sylvester equation
\begin{equation}
\alpha q-q\beta=\Delta
\label{4.5}
\end{equation}
has a unique solution $q=(\overline{\alpha}\Delta-\Delta\beta)\cX_{[\alpha]}(\beta)^{-1}$ for 
any $\Delta\in\bH$. Consequently, if $f\in\bH[z]$ and $\alpha\not\sim\beta$, then 
\begin{equation}
(L_\alpha f)^{\br}(\beta)=(\overline{\alpha}(f^{\bl}(\alpha)-f^{\br}(\beta))-
(f^{\bl}(\alpha)-f^{\br}(\beta))\beta)\cX_{[\alpha]}(\beta)^{-1}.
\label{4.6}
\end{equation}
\label{L:6.1} 
\end{lemma}
\begin{proof}
Multiplying both parts of \eqref{4.5} by $\overline{\alpha}$ on the left and by  $\beta$ on the 
right gives
$$
|\alpha|^2q-\overline{\alpha}q\beta=\overline{\alpha}\Delta\quad\mbox{and}\quad
\alpha q\beta-q\beta^2=\Delta\beta,
$$
respectively. Subtracting the second equation from the first and commuting real coefficients 
we get
\begin{align}
\overline{\alpha}\Delta-\Delta\beta&=
|\alpha|^2q-(\alpha+\overline{\alpha})q\beta+q\beta^2\notag\\
&=q(|\alpha|^2-2\beta{\rm Re}(\alpha)+\beta^2)=q\cX_{[\alpha]}(\beta),\label{4.7}
\end{align}
and the desired formula for $q$ follows since $\alpha\not\sim\beta$ and therefore, 
$\cX_{[\alpha]}(\beta)\neq 0$. Applying this formula to the equation \eqref{4.4}
(i.e., for $\Delta=f^{\bl}(\alpha)-f^{\br}(\beta)$), we get \eqref{4.6}.
\end{proof}
The case where $\alpha$ and $\beta$ are equivalent is more interesting.
Let us denote by $\mathbb S$ the unit sphere of
purely imaginary quaternions. Any $I\in\mathbb S$ is such that $I^2=-1$. Recall that
if $\alpha$ and $\beta$ are two equivalent quaternions, then due to characterization \eqref{2.1},
they can be written in the form
\begin{equation}
\alpha=x+y I,\quad \beta=x+y \widetilde{I} \qquad (x\in\R, \; y>0, \; I, \widetilde{I}\in\mathbb S).
\label{4.8}
\end{equation}
Since $\bH$ is a (four-dimensional) vector space over $\R$, we may define
orthogonal complements with respect to the usual euclidean metric in $\R^4$.
For $\alpha$ and $\beta$ as in \eqref{4.8}, we define the plane (the two-dimensional
subspace of $\bH\cong\R^4$) $\Pi_{\alpha,\beta}$ via the formula
\begin{equation}
\Pi_{\alpha,\beta}=\left\{\begin{array}{lll}
span \{1,I\}=\{u+vI: \, u,v\in\R\},&\mbox{if}& \beta=\alpha,\\
\left(span \{1,I\}\right)^\perp,&\mbox{if}& \beta=\overline{\alpha},\\ 
span \{I+\widetilde{I}, \;  1-I\widetilde{I}\},&\mbox{if}& \beta\neq \alpha,\overline{\alpha}.  
\end{array}\right.
\label{4.9}
\end{equation}
Since $\overline{\alpha}=x-y I$, it follows that
$\Pi_{\overline{\alpha},\overline{\alpha}}=\Pi_{\alpha,\alpha}$,
$\Pi_{\overline{\alpha},\alpha}=\Pi_{\alpha, \overline{\alpha}}$ and
\begin{equation}
\Pi_{\overline{\alpha},\beta}=span \{I-\widetilde{I}, \;  1+I\widetilde{I}\}\quad
\mbox{if}\quad\beta\neq \alpha,\overline{\alpha}.
\label{4.10}
\end{equation}
\begin{lemma}
Let $\alpha\sim\beta$ be of the form \eqref{4.8}. Then the Sylvester equation \eqref{4.2} has a
solution
if and only if $\Delta\in\Pi_{\overline{\alpha},\beta}$ or equivalently, if and only if
\begin{equation}
\overline{\alpha}\Delta=\Delta\beta.
\label{4.11}
\end{equation}
If this is the case, the solution set for the equation \eqref{4.5} is the affine plane
$$
(2{\rm Im}(\alpha))^{-1}\Delta+\Pi_{\alpha,\beta}=-\Delta(2{\rm
Im}(\beta))^{-1}+\Pi_{\alpha,\beta}.
$$
\label{L:6.2}
\end{lemma}   
\begin{proof}
If $\alpha\sim\beta$, then
$\cX_{[\alpha]}(\beta)=0$, and calculation \eqref{4.7} shows that we necessarily have
\eqref{4.11}. Substituting \eqref{4.8} into \eqref{4.11} gives $-yI\Delta=y\Delta \widetilde{I}$,
which is equivalent (since $y\neq 0$ and $I^2=-1$) to $\Delta=I\Delta\widetilde{I}$. Therefore,
$$
\alpha I\Delta-I\Delta\beta=yI^2\Delta-yI\Delta \widetilde{I}=-2y\Delta
$$
and hence, the element $q_0=(2{\rm Im}(\alpha))^{-1}\Delta=-\Delta(2{\rm
Im}(\beta))^{-1}$ is a particular 
solution of the equation \eqref{4.5}. It 
remains to show that the solution set of the homogeneous Sylvester equation
\begin{equation}
\alpha p-p\beta=0
\label{4.12}
\end{equation}
coinsides with the plane $\Pi_{\alpha,\beta}$ defined in \eqref{4.9}. To this end, we first 
observe
that since $\R$ is the center of $\bH$, it follows that the solution set $\Omega_{\alpha,\beta}$
of the equation \eqref{4.12} is a (real) subspace of $\bH\cong\R^4$.
For a given $I\in\mathbb S$ and any $J\in\mathbb S$ which is orthogonal
to $I$ (as a vector in $\bH\cong\R^4$), the elements $\{{\bf 1}, I, J, IJ\}$
form an orthonormal basis in $\bH$ and therefore, any element $p\in\bH$ admits a unique
representation
\begin{equation}
p=x_0+x_1 I+x_2J+x_3 IJ \qquad (x_0,x_1,x_2,x_3\in\mathbb R)
\label{4.13}  
\end{equation}
similar to \eqref{1.4}. Observe that since $I,J\in\mathbb S$ and $I\perp J$, we have  
$IJ=-JI$.  We then use the latter equality along with representations \eqref{4.8} and 
\eqref{4.13} to
compute
$$
\alpha p-p\alpha=2y(x_2IJ-x_3J)\quad\mbox{and}\quad \alpha p-p\overline{\alpha}=2y(x_0I-x_1).
$$
Thus, $\alpha p=p\alpha$ if and only if $x_2=x_3=0$ and hence, $p\in span \{1,I\}$,  
and on the other hand, $\alpha p=p\overline{\alpha}$  if and only if $x_0=x_1=0$ and hence,
$p\in span \{J,IJ\}=\left(span \{1,I\}\right)^\perp$, which proves that $\Omega_{\alpha,\beta}$
indeed is equal to the plane $\Pi_{\alpha,\beta}$ for the cases where $\beta=\alpha$ or
$\beta=\overline{\alpha}$.
  
\smallskip

For the remaining case, we will argue as follows. Since $\beta\neq
\alpha,\overline{\alpha}$, representations \eqref{4.8} hold with $\widetilde{I}\neq \pm I$.
Letting $p_1=I+\widetilde{I}$ and $p_2=1-I\widetilde{I}$ we see that
\begin{align*}
\alpha p_1-p_1\beta&=y(I^2+I\widetilde{I}-I\widetilde{I}-\widetilde{I}^2)=0,\\
\alpha p_2-p_2\beta&=y(I-I^2\widetilde{I}-\widetilde{I}+I\widetilde{I}^2)=0,
\end{align*}
since $I^2=\widetilde{I}^2=-1$. Thus, $p_1$ and $p_2$ are linearly independent (over $\R$) solutions of
the equation \eqref{4.12} and therefore $\Omega_{\alpha,\beta}\supset \Pi_{\alpha,\beta}$, so 
that
$\dim \Omega_{\alpha,\beta}\ge \dim \Pi_{\alpha,\beta}=2$.

\smallskip

Similarly, one can verify that  $p_3=I-\widetilde{I}$ and $p_4=1+I\widetilde{I}$ are two linear
independent solutions to  the Sylvester equation $\overline{\alpha}p=p\beta$, the solution set
$\Omega_{\overline{\alpha},\beta}$ of which contains the plane $\Pi_{\overline{\alpha},\beta}$ 
(see \eqref{4.10}) and
therefore, is a subspace of $\bH$ of dimension of at least two. Observe that since $\alpha\neq
\overline{\alpha}$, equalities $\alpha p=p\beta=\overline{\alpha}p$ imply $p=0$.
Therefore $\Omega_{\alpha,\beta}\cap \Omega_{\overline{\alpha},\beta}=\{0\}$, and consequently,
$\dim \Omega_{\alpha,\beta}=\dim \Omega_{\overline{\alpha},\beta}=2$. Therefore,
$\Omega_{\alpha,\beta}=\Pi_{\alpha,\beta}$ and
$\Omega_{\overline{\alpha},\beta}=\Pi_{\overline{\alpha},\beta}$. In particular, $\Delta$   
is subject to condition \eqref{4.11} if and only if it belongs to 
$\Pi_{\overline{\alpha},\beta}$. 
\end{proof}
\begin{corollary}
For $f\in\bH[z]$ and $\alpha\sim\beta\in\bH$, 
$$\overline{\alpha}\left(f^{\bl}(\alpha)-f^{\br}(\beta)\right)=\left(f^{\bl}(\alpha)-
f^{\br}(\beta)\right)\beta.
$$
\label{C:4.4}
\end{corollary}
\begin{proof}
Equality \eqref{4.4} tells us that for $\Delta:=f^{\bl}(\alpha)-f^{\br}(\beta)$, 
the Sylvester equation \eqref{4.5} has a solution. By Lemma \eqref{L:6.2}, equality 
\eqref{4.11} holds, which is the same as \eqref{4.8}, due to the present choice of $\Delta$.
\end{proof}
We may now present necessary and sufficient conditions for the  two-sided Lagrange problem 
to have a solution.
\begin{theorem}
Assume that none three elements of the set $\Lambda\bigcup\Omega$ are equivalent.
There is a polynomial $f\in\bH[z]$ satisfying conditions \eqref{1.8}, \eqref{1.9} if and only if 
$\overline{\alpha}_i\left(c_i-d_j\right)=\left(c_i-d_j\right)\beta_j$ for each pair 
$(\alpha_i,\beta_j)\in\Lambda\times\Omega$ of equivalent nodes.
\label{T:4.7}
\end{theorem}
The ``only if" part follows from Corollary \ref{C:4.4}. The sufficiency part will be confirmed 
in the next section.

\section{Two-sided problem}
\setcounter{equation}{0}

We still assume that none three of interpolation nodes are equivalent.
To be more specific, we assume that there are $k$ equivalent pairs 
in $\Lambda\times\Omega$ (the case $k=0$ is not excluded), and we rearrange the sets $\Lambda$ and $\Omega$ 
\eqref{1.7} 
so that these equaivalent pairs are $(\alpha_i,\beta_i)$ for $i=1,\ldots,k$.
In other words,
\begin{equation}
\alpha_i\sim\beta_i \; \; (1\le i\le k); \; \; 
[\alpha_i]\cap \Omega=\emptyset
\; \;  (k<i\le n); \; \;  [\beta_j]\cap \Lambda=\emptyset \; \; (k<j\le m).
\label{5.1}
\end{equation}
We also assume that the necessary conditions from Theorem \ref{T:4.7} hold:
\begin{equation}
\overline{\alpha}_i\left(c_i-d_i\right)=\left(c_i-d_i\right)\beta_i \qquad (i=1,\ldots,k).
\label{5.2}
\end{equation}
One may try to handle the two-sided problem in a standard way by combining the
explicit formula for the particular low degree solution and the parametrization 
of the solution set for the homogeneous problem. Both ingredients are not as simple as
in one-sided cases.
\subsection{Homogeneous problems}The homogeneous counterpart of 
the two-sided Lagrange problem 
\eqref{1.8}, \eqref{1.9} consists of finding all $f\in\bH[z]$ such that 
\begin{equation}
f^{\bl}(\alpha_i)=0 \quad (1\le i\le n);\qquad
f^{\br}(\beta_j)=0\quad (1\le j\le m).
\label{5.3}  
\end{equation}
The set ${\mathbb I}_{\Lambda,\boldsymbol \ell}$ of all polynomials $f\in\bH[z]$ satisfying left
conditions in \eqref{5.3} is the right ideal generated by the {\bf lmp} $P_{\Lambda,\boldsymbol\ell}$
of $\Lambda=\{\alpha_1,\ldots,\alpha_n\}$; in fact, ${\mathbb I}_{\Lambda,\boldsymbol \ell}$ is the finite 
intersection of maximal right ideals
in $\bH[z]$:
\begin{equation}
{\mathbb I}_{\Lambda,\boldsymbol \ell}=\bigcap_{i=1}^n\left\{f\in\bH[z]: \;
f(\alpha_i)=0\right\}=P_{\Lambda,\boldsymbol\ell}\cdot \bH[z].
\label{5.4}
\end{equation}
Analogously, the set ${\mathbb I}_{\Omega,\bf r}$ of all polynomials satisfying right
conditions in \eqref{5.3} is the left ideal (in fact, the finite intersection  of maximal  left ideals) in 
$\bH[z]$ generated by the {\bf rmp} $P_{\Omega,\bf r}$ of the set $\Omega=\{\beta_1,\ldots,\beta_m\}$:
\begin{equation}
{\mathbb I}_{\Omega,\bf r}=\bigcap_{j=1}^m\left\{f\in\bH[z]: \; f(\beta_j)=0\right\}=\bH[z]\cdot P_{\Omega,\bf r}.
\label{5.5}  
\end{equation}
Hence, ${\mathbb I}_{\Lambda,\boldsymbol \ell}\bigcap {\mathbb I}_{\Omega,\bf r}$ is the solution set to the 
problem \eqref{5.3} and the next question is to describe this intersection analytically. The next result shows 
that in case $k=0$ in \eqref{5.1}, ${\mathbb I}_{\Lambda,\boldsymbol \ell}\bigcap {\mathbb I}_{\Omega,\bf r}=
P_{\Lambda,\boldsymbol\ell}\cdot \bH[z]\cdot P_{\Omega,\bf r}$. 
\begin{theorem} 
A polynomial $f\in\bH[z]$ satisfies \eqref{5.3} and additional conditions 
\begin{equation}
(L_{\alpha_i} f)^{\br}(\beta_i)=0\quad\mbox{for}\quad i=1,\ldots,k
\label{5.6}   
\end{equation}
if and only if it belongs to $P_{\Lambda,\boldsymbol\ell}\cdot \bH[z]\cdot P_{\Omega,\bf r}$.
\label{T:5.1}
\end{theorem}
\begin{proof} For any $h\in\bH[z]$, the polynomial $f=P_{\Lambda,\boldsymbol\ell}\cdot h\cdot P_{\Omega,\bf 
r}$ clearly satisfies conditions \eqref{5.3}. Since $P_{\Lambda,\boldsymbol\ell}^{\bl}(\alpha_i)=0$, we 
have 
\begin{equation}
L_{\alpha_i} f=L_{\alpha_i}(P_{\Lambda,\boldsymbol\ell}\cdot h\cdot P_{\Omega,\bf  
r})=(L_{\alpha_i}P_{\Lambda,\boldsymbol\ell})\cdot h\cdot P_{\Omega,\bf r}
\label{5.7}
\end{equation}
and since $P_{\Omega,\bf r}^{\br}(\beta_i)=0$, the right evaluation at $z=\beta_i$ applied to both sides of 
\eqref{5.7} implies \eqref{5.6}.

\smallskip

Conversely, let $f\in\bH[z]$ satisfy conditions \eqref{5.3}, \eqref{5.6}. Due to the left conditions 
in \eqref{5.3}, $f$ is in ${\mathbb I}_{\Lambda,\boldsymbol \ell}$ and thus, it is of the form 
$f=P_{\Lambda,\boldsymbol\ell}\cdot g$ for some $g\in\bH[z]$. To complete the proof, it suffices to show that 
$g\in {\mathbb I}_{\Omega,{\bf r}}$. 

\smallskip

If $j>k$, (i.e., if $[\beta_j]\cap 
\Lambda=\emptyset$; see \eqref{5.1}), then $P_{\Lambda,\boldsymbol\ell}$ does not have zeros
(either left or right, by Remark \eqref{R:2.6a}) in the conjugacy class $[\beta_j]$. Then the equality
$$
f^{\br}(\beta_j)=(P_{\Lambda,\boldsymbol\ell}\cdot g)^{\br}(\beta_j)=0
$$ 
implies $g^{\br}(\beta_j)=0$
since otherwise, the polynomial $P_{\Lambda,\boldsymbol\ell}$ had a right zero at 
\begin{equation}
g^{\br}(\beta_j)\beta_jg^{\br}(\beta_j)^{-1}\in[\beta_j]
\label{5.7a}   
\end{equation}
which is a contradiction. Thus, $g^{\br}(\beta_j)=0$ for all $j=k+1,\ldots, m$.

\smallskip

Observe that $\cZ_{\boldsymbol\ell}(L_{\alpha_j}P_{\Lambda,\boldsymbol\ell})\bigcup 
\cZ_{\bf r}(L_{\alpha_j}P_{\Lambda,\boldsymbol\ell})\subset \bigcup_{i\neq j}[\alpha_i]$. Hence, 
if $j\le k$ (i.e., $\alpha_j\sim\beta_j$  and $\alpha_i\not\sim \alpha_j$ whenever $i\neq j$), 
then we have
\begin{equation}
(\cZ_{\boldsymbol\ell}(L_{\alpha_j}P_{\Lambda,\boldsymbol\ell})\cup
\cZ_{\bf r}(L_{\alpha_j}P_{\Lambda,\boldsymbol\ell}))\cap [\alpha_j]=\emptyset.
\label{5.8}
\end{equation}
Then the equality 
$$
(L_{\alpha_j} f)^{\br}(\beta_j)=(L_{\alpha_j}(P_{\Lambda,\boldsymbol\ell}\cdot g))^{\br}(\beta_j)
=((L_{\alpha_j}P_{\Lambda,\boldsymbol\ell})\cdot g)^{\br}(\beta_j)=0
$$
implies $g^{\br}(\beta_j)=0$, since otherwise, $L_{\alpha_j} P_{\Lambda,\boldsymbol\ell}$ has 
right zero at the point \eqref{5.7a} which 
contradicts \eqref{5.8}. We thus have $g^{\br}(\beta_j)=0$ for all $j=1,\ldots, m$ and hence, 
$g\in {\mathbb I}_{\Omega,{\bf r}}$ which completes the proof. 
\end{proof}
If $k>0$ in \eqref{5.1}, then the set $P_{\Lambda,\boldsymbol\ell}\cdot \bH[z]\cdot P_{\Omega,\bf r}$
is properly included in ${\mathbb I}_{\Lambda,\boldsymbol \ell}\bigcap {\mathbb I}_{\Omega,\bf r}$.
The analytic description of ${\mathbb I}_{\Lambda,\boldsymbol \ell}\bigcap {\mathbb I}_{\Omega,\bf r}$
is given below. Recall that $P_{{\Lambda}_i,{\boldsymbol \ell}}$ is the {\bf lmp} of the set
$\Lambda_i=\Lambda\backslash\{\alpha_i\}$ and $P_{{\Omega}_i,{\bf r}}$ is the {\bf rmp} of 
$\Omega_i=\Omega\backslash\{\beta_i\}$.
\begin{theorem}
Under assumptions \eqref{5.1}, let 
\begin{equation}
\widetilde{\alpha}_i=P_{{\Lambda}_i,{\boldsymbol \ell}}^{\bl}(\alpha_i)^{-1}\cdot\alpha_i\cdot
P_{{\Lambda}_i,{\boldsymbol \ell}}^{\bl}(\alpha_i),\qquad
\widetilde{\beta}_i=P_{{\Omega}_i,{\bf r}}^{\br}(\beta_i)\cdot\beta_i\cdot
P_{{\Omega}_i,{\bf r}}^{\br}(\beta_i)^{-1},
\label{5.9}
\end{equation}
so that $\widetilde{\alpha}_i\sim \alpha_i\sim\beta_i\sim \widetilde{\beta}_i$.
Let $\Pi_{\widetilde{\alpha}_i,\widetilde{\beta}_i}$ be the plane defined via formula \eqref{4.9} 
for 
$i=1,\ldots,k$.
Then $f$ belongs to ${\mathbb I}_{\Lambda,\boldsymbol \ell}\bigcap 
{\mathbb I}_{\Omega,\bf r}$ if and only if it is of the form 
\begin{equation}
f(z)=\sum_{i=1}^k P_{\Lambda,\boldsymbol\ell}(z)\cdot \mu_i \cdot P_{{\Omega}_i,{\bf r}}(z)+
P_{\Lambda,\boldsymbol\ell}(z)\cdot h(z)\cdot P_{{\Omega},{\bf r}}(z)  
\label{5.10}
\end{equation}
for some $h\in\bH[z]$ and $\mu_i\in \Pi_{\widetilde{\alpha}_i,\widetilde{\beta}_i}$.
\label{T:5.2}
\end{theorem}
The proof will be given in Section 6 as a consequence of Theorem \ref{T:5.7}. 
Although the result of Theorem \ref{T:5.2} will not be used 
in our further analysis, it is of some independent interest as we now explain. 
\subsection{Quasi-ideals and principal bi-ideals in $\bH[z]$} The notions of quasi-ideals
and bi-ideals in associative rings were introduced in \cite{stein} and \cite{laj}, respectively.
We recall these notions in the present context
of $R=\bH[z]$ (a ring with identity, in which any ideal is principal).

\smallskip

A subset of $\bH[z]$ is called a {\em bi-ideal} if it is a left ideal of some right ideal in $\bH[z]$
or equivalently, it is a right ideal of a left ideal in $\bH[z]$. Although all ideals in 
$\bH[z]$ are principal, a left ideal of a right ideal of $\bH[z]$ is not principal, in general. Let us say 
that $Q$ is a {\em principal bi-ideal} (this notion is not common) if it is a principal left ideal of some 
right ideal in $\bH[z]$. According to this definition, each principal bi-ideal in $\bH[z]$ is of the form 
$Q=p\cdot \bH[z]\cdot q$ for some fixed $p,q\in\bH[z]$ and hence, a principal bi-ideal can be 
equivalently defined as  a principal right ideal of some left ideal in $\bH[z]$. Theorem \ref{T:5.1}
tells us that the solution set of the homogeneous interpolation problem \eqref{5.3}, \eqref{5.6}
is a principal bi-ideal.

\smallskip

Recall that a subset of a ring $R$ with identity is called a {\em quasi-ideal} if it is the intersection of a 
left ideal of $R$ with a right ideal of $R$. Theorem \ref{5.2} provides an analytic description of a 
quasi-ideal in $\bH[z]$ which is the intersection of finitely many maximal left ideals and 
finitely many maximal right ideals in $\bH[z]$. 

\smallskip 

Quasi-ideals and principal bi-ideals are special instancies of bi-ideals. In general, these 
two notions are distinct. To demonstrate this, let
$$
{\mathbb I}_{\boldsymbol\ell}=(z-{\bf i})\cdot\bH[z],\quad
{\mathbb I}_{\bf r}=\bH[z]\cdot(z-{\bf j}),\quad Q=(z-{\bf i})\cdot {\mathbb I}_{\bf r}={\mathbb 
I}_{\boldsymbol\ell}\cdot(z-{\bf j}),
$$
and let us show that 
\begin{enumerate}
\item the principal bi-ideal $Q$ is not a quasi-ideal and that 
\item the quasi-ideal
${\mathbb I}_{\boldsymbol\ell}\cap {\mathbb I}_{\bf r}$ is not a principal bi-ideal.
\end{enumerate}
We have $Q\subset {\mathbb I}_{\boldsymbol\ell}\cap {\mathbb I}_{\bf r}$, and the inclusion is proper
since the function
\begin{equation}
(z-{\bf i})({\bf i}+{\bf j})=({\bf i}+{\bf j})(z-{\bf j}) \quad\mbox{belongs to}\quad {\mathbb 
I}_{\boldsymbol\ell}\cap {\mathbb I}_{\bf r},
\label{5.11}   
\end{equation}
but not to $Q$.  Let us assume that there is a proper right ideal 
$\widetilde{{\mathbb I}}_{\boldsymbol\ell}\subset {\mathbb I}_{\boldsymbol\ell}$ containing 
$Q$. Then $\widetilde{{\mathbb I}}_{\boldsymbol\ell}$ is generated by a right multiple of $(z-{\bf i})$
and hence $L_{\bf i}Q={\mathbb I}_{\bf r}\subset (z-\alpha)\cdot \bH[z]$ for some $\alpha\in\bH$. Since 
$h(z)=z-{\bf j}$ belongs to 
${\mathbb I}_{\bf r}$, we necessarily have $\alpha={\bf j}$. Since $g(z)={\bf i}(z-{\bf j})=(z+{\bf j}){\bf 
i}$ belongs to ${\mathbb I}_{\bf r}$, we also have $\alpha=-{\bf j}$ which is a contradiction. 
Since the intersection of right ideals is a right ideal, it follows that any right ideal in $\bH[z]$ 
containing $Q$ also contains ${\mathbb I}_{\boldsymbol\ell}$. Similarly, any left ideal in $\bH[z]$ 
containing $Q$ also contains ${\mathbb I}_{\bf r}$. 
Therefore, ${\mathbb I}_{\boldsymbol\ell}\cap {\mathbb I}_{\bf r}$ is the minimal quasi-ideal
containing $Q$, and since $Q\neq {\mathbb I}_{\boldsymbol\ell}\cap {\mathbb I}_{\bf r}$, it follows that
$Q$ is not a quasi-ideal.

\smallskip

To show part (2), we first observe that ${\mathbb I}_{\boldsymbol\ell}\cap {\mathbb I}_{\bf r}$
is not a left ideal in $\bH[z]$. Indeed, if ${\mathbb I}_{\boldsymbol\ell}\cap {\mathbb I}_{\bf 
r}$ were a left ideal, it would have been a proper ideal of ${\mathbb I}_{\bf r}$ generated therefore 
by a polynomial of degree at least two and not containing therefore, polynomials of degree one.
The latter contradicts to \eqref{5.11}. Similarly, ${\mathbb I}_{\boldsymbol\ell}\cap {\mathbb I}_{\bf r}$
is not a right ideal either. If ${\mathbb I}_{\boldsymbol\ell}\cap {\mathbb I}_{\bf r}$ is a
principal bi-ideal, it admits a representation ${\mathbb I}_{\boldsymbol\ell}\cap {\mathbb I}_{\bf r}
=p\cdot \bH[z]\cdot q$ for some polynomials $p,q$ of degree at least one. Then it again cannot contain 
polynomials of degree one which contradicts to \eqref{5.11}. Therefore, ${\mathbb I}_{\boldsymbol\ell}\cap 
{\mathbb I}_{\bf r}$ is not a principal bi-ideal.
\subsection{Elementary cases} In formula \eqref{1.2}, the complex Lagrange interpolation polinomial  
$\widetilde{f}$ is constructed as the sum of polynomials $\widetilde{f}_j(z)=\frac{c_jp_j(z)}{p_j(z_j)}$
satisfying interpolation conditions  $\widetilde{f}_j(z_j)=c_j$ and $\widetilde{f}_j(z_i)=0$ for $i\neq j$.
The formulas \eqref{3.8} for the left Lagrange polynomial $\widetilde{f}_{\boldsymbol\ell}$ and 
\eqref{3.10} for the right Lagrange polynomial $\widetilde{f}_{\bf r}$ followed the same 
strategy:
to construct ``elementary" polynomials satisfying one non-homogeneous interpolation condition
from \eqref{1.8} (respectively, from \eqref{1.9}) and having left (respectively, right) zeros at 
all other interpolation nodes, and then to construct $\widetilde{f}_{\boldsymbol\ell}$ (respectively, 
$\widetilde{f}_{\bf r}$) as the sum of these polynomials. In this section we will adapt this approach to the 
two-sided problem.  We start with a techical result.
\begin{lemma}
Given $P\in\bH[z]$, given $\beta\not\in\cZ(P^{\sharp}P)$ and non-zero $d,\delta\in\bH$,
\begin{align}
&d=P^{\br}(\delta\beta\delta^{-1})\cdot\delta \; \; \Longleftrightarrow \; \;
\delta=P^{\sharp\br}(d\beta d^{-1})\cdot d\cdot(P^{\sharp}P)(\beta)^{-1};
\label{5.12} \\
&d=\delta\cdot P^{\bl}(\delta^{-1}\beta\delta) \; \; \Longleftrightarrow \; \; 
\delta=(P^{\sharp}P)(\beta)^{-1}\cdot d\cdot 
P^{\sharp\bl}(d^{-1}\beta d).
\label{5.13}
\end{align}
\label{L:7.1} 
\end{lemma} 
\begin{proof} If $d=P^{\br}(\delta\beta\delta^{-1})\cdot\delta$, then 
\begin{align*}
P^{\sharp\br}(d\beta d^{-1})\cdot d&=P^{\sharp\br}\left(P^{\br}(\delta\beta\delta^{-1})\delta
\beta \delta^{-1}P^{\br}(\delta\beta\delta^{-1})^{-1}\right)P^{\br}(\delta\beta\delta^{-1})\delta\\
&=(P^{\sharp}P)^{\br}(\delta\beta\delta^{-1})\cdot\delta=\delta\cdot 
(P^{\sharp}P)(\beta),
\end{align*}
where the second equality follows from formula \eqref{2.7}, and the 
third equality follows since the polynomial $P^{\sharp}P$ is real. Since $(P^{\sharp}P)(\beta)\neq 
0$, the latter formula implies the formula for $\delta$ in \eqref{5.12} which completes the proof of 
implication $\Rightarrow$ in \eqref{5.12}.  To prove the reverse implication, we observe that 
the right equality in \eqref{5.12} is equivalent to 
$$
\delta\cdot(P^{\sharp}P)(\beta)=
P^{\sharp\br}(d\beta d^{-1})\cdot d.
$$
We then apply the implication $\Rightarrow$ (just proven) to the latter equality, i.e.,
to $P^\sharp$, $\delta\cdot(P^{\sharp}P)(\beta)$ and $d$ instead of 
$P$, $d$ and $\delta$, respectively:
\begin{align*}
d&=P^{\br}\left(\delta\cdot (P^{\sharp}P)(\beta)\cdot \beta \cdot
(P^{\sharp}P)(\beta)^{-1}\cdot \delta^{-1}\right)\cdot\delta
\cdot (P^{\sharp}P)(\beta)\cdot(P^{\sharp}P)(\beta)^{-1}\\
&=P^{\br}\left(\delta \beta \delta^{-1}\right)\cdot\delta,
\end{align*}
where the second equality holds since $\beta$ and $(P^{\sharp}P)(\beta)$ commute.
This completes the proof of the equivalence \eqref{5.12}. If we apply this 
equivalence to $P^\sharp$ and quaternionic conjugates of $\beta,d,\delta$, we get
$$
\overline{d}=P^{\sharp\br}\left(\overline{\delta^{-1}\beta\delta}\right)\cdot\overline{\delta}
\; \; \Longleftrightarrow \; \;
\overline{\delta}=P^{\br}\left(\overline{d^{-1}\beta d}\right)\cdot \overline{d}\cdot 
\left[(PP^{\sharp})(\overline{\beta})\right]^{-1}.
$$
Taking quaternionic conjugates in both equalities and making use of equality 
$f^{\bl}(\alpha)=\overline{f^{\sharp\br}(\overline{\alpha})}$ (holding for all $\alpha\in\bH$,
due to \eqref{2.2}), we arrive at \eqref{5.13}.
\end{proof}
The next three lemmas present ``elementary" polynomials which then will be used to 
construct a family of low-degree solutions to the problem \eqref{1.8}, \eqref{1.9}.
\begin{lemma}
Under assumptions \eqref{5.1}, let $s\in\{k+1,\ldots,m\}$ and let 
\begin{equation}
\gamma_s=\left\{\begin{array}{ccc} P_{\Lambda,{\boldsymbol\ell}}^{\sharp\br}(d_s\beta_sd_s^{-1})\cdot d_s\cdot 
(P_{\Lambda,{\boldsymbol\ell}}^{\sharp}P_{\Lambda,{\boldsymbol\ell}})(\beta_s)^{-1}\cdot
P_{{\Omega}_s,{\bf r}}^{\br}(\beta_s)^{-1},&\mbox{if}& d_s\neq 0,\\
0,&\mbox{if}& d_s=0.\end{array}\right.
\label{5.14}
\end{equation}
Then 
\begin{equation}
\widetilde{f}_{{\bf r},s}(z)=P_{\Lambda,{\boldsymbol\ell}}(z)\cdot \gamma_s\cdot P_{{\Omega}_s,{\bf r}}(z)
\label{5.15}
\end{equation}
is a unique polynomial of degree less than $n+m$ satisfying conditions
\begin{equation}
f^{\br}(\beta_s)=d_s,\quad f^{\br}(\beta_i)=0 \; \; (i\neq j),\quad f^{\bl}(\alpha_i)=0 \; \; (i=1,\ldots,n),
\label{5.16} 
\end{equation}
\begin{equation}
(L_{\alpha_i} f)^{\br}(\beta_i)=0\quad\mbox{for}\quad i=1,\ldots,k. 
\label{5.17}   
\end{equation}
\label{L:5.4}
\end{lemma}
\begin{proof} If $d_s=0$, we have the homogeneous interpolation problem \eqref{5.3}, \eqref{5.6}, and 
the statement follows from Theorem \ref{T:5.1}. Assume that $d_s\neq 0$.  
Since none three elements in $\Omega$ are equivalent, the set
$\cZ_{\bf r}(P_{{\Omega}_s,{\bf r}})={\Omega}_s=\Omega\backslash\{\beta_s\}$  
contains at most one conjugate  of $\beta_s$; therefore, $P_{{\Omega}_s,{\bf r}}(\beta_s)\neq 0$.
Since $s>k$, we have by assumption \eqref{5.1},
$\beta_s\not\in\bigcup_{i=1}^n[\alpha_i]=\cZ(P_{\Lambda,{\boldsymbol\ell}}^{\sharp}P_{\Lambda,{\boldsymbol\ell}})$
and therefore, $(P_{\Lambda,{\boldsymbol\ell}}^{\sharp}P_{\Lambda,{\boldsymbol\ell}})(\beta_s)\neq 0$.
Hence, the formula \eqref{5.14} makes sense. 

\smallskip

To show that  $\widetilde{f}_{{\bf r},s}$ defined as in \eqref{5.15} satisfies the first condition in 
\eqref{5.16}, let write \eqref{5.14} (recall that $d_s\neq 0$) equivalently as 
\begin{equation}
\gamma_s\cdot P_{{\Omega}_s,{\bf r}}^{\br}(\beta_s)
=P_{\Lambda,{\boldsymbol\ell}}^{\sharp\br}(d_s\beta_sd_s^{-1})\cdot d_s\cdot
(P_{\Lambda,{\boldsymbol\ell}}^{\sharp}P_{\Lambda,{\boldsymbol\ell}})(\beta_s)^{-1}.
\label{5.17a}
\end{equation}
By implication $\; \Leftarrow\;$ in \eqref{5.12} (with $\delta=\gamma_s\cdot P_{{\Omega}_s,{\bf 
r}}^{\br}(\beta_s)$), we then have from \eqref{5.17a}
$$
d_s=P_{\Lambda,{\boldsymbol\ell}}^{\br}\left(
\gamma_s \cdot P_{{\Omega}_s,{\bf r}}^{\br}(\beta_s) \cdot \beta_s\cdot
P_{{\Omega}_s,{\bf r}}^{\br}(\beta_s)^{-1}\cdot \gamma_s^{-1}
\right)\cdot \gamma_j\cdot P_{{\Omega}_s,{\bf r}}^{\br}(\beta_s).
$$
On the other hand, applying formula \eqref{2.7} to the product in \eqref{5.15} gives
$$
\widetilde{f}_{{\bf r},s}^{\br}(\beta_s)=P_{\Lambda,{\boldsymbol\ell}}^{\br}\left(
\gamma_s\cdot P_{{\Omega}_s,{\bf r}}^{\br}(\beta_s)\cdot  \beta_s\cdot  
P_{{\Omega}_s,{\bf r}}^{\br}(\beta_s)^{-1}\cdot \gamma_s^{-1}
\right)\cdot \gamma_s\cdot P_{{\Omega}_s,{\bf r}}^{\br}(\beta_s).
$$
The two last equalities imply $\widetilde{f}_{{\bf r},s}^{\br}(\beta_s)=d_s$ so that
$\widetilde{f}_{{\bf r},s}$ indeed satisfies the first condition in \eqref{5.15}.
Other conditions in \eqref{5.15} are met since $P_{\Lambda,{\boldsymbol\ell}}^{\bl}(\alpha_i)=0$
(for all $\alpha_i\in\Lambda$) and $P_{{\Omega}_s,{\bf r}}^{\br}(\beta_i)=0$ (for all $\beta_i\in
\Omega\backslash\{\beta_s\}$) by definitions of left and right minimal polynomials. Furthermore, since
$P_{\Lambda,{\boldsymbol\ell}}^{\bl}(\alpha_i)=0$, we have, by \eqref{4.2} and \eqref{5.15},
$$
L_{\alpha_i}\widetilde{f}_{{\bf r},s}=
(L_{\alpha_i}P_{\Lambda,{\boldsymbol\ell}})\cdot \gamma_s\cdot P_{{\Omega}_s,{\bf r}}
$$  
and since $P_{{\Omega}_s,{\bf r}}^{\br}(\beta_i)=0$ for all $i=1,\ldots,k$, equalities
\eqref{5.17} hold.

\smallskip

It is clear from \eqref{5.15} that $\deg(\widetilde{f}_{{\bf r},s})=
\deg(P_{\Lambda,{\boldsymbol\ell}})+\deg(P_{{\Omega}_j,{\bf r}})=n+m-1$ (since $\gamma_s\neq 0$).
If $f$ is any polynomial subject to conditions  \eqref{5.16}, \eqref{5.17}, then the polynomial 
$f-\widetilde{f}_{{\bf r},s}$  belongs to $P_{\Lambda,{\boldsymbol\ell}}\cdot \bH[z]\cdot P_{\Omega,{\bf 
r}}$ (by Theorem \ref{5.1}) and therefore, either $f\equiv \widetilde{f}_{{\bf r},s}$ or 
$\deg (f-\widetilde{f}_{{\bf r},s})\ge 
\deg (P_{\Lambda,{\boldsymbol\ell}})+\deg(P_{\Omega,{\bf r}})=m+n$. This implies the uniqueness of a low-degree solution. 
\end{proof}
\begin{lemma}
Under assumptions \eqref{5.1}, let , let $s\in\{k+1,\ldots,n\}$ and let
\begin{equation}
\rho_s=\left\{\begin{array}{ccc}P_{{\Lambda}_s,{\boldsymbol \ell}}^{\bl}(\alpha_s)^{-1}
\cdot(P_{\Omega,{\bf r}}^{\sharp}P_{\Omega,{\bf r}})(\alpha_s)^{-1}\cdot c_s\cdot 
P_{\Omega,{\bf r}}^{\sharp\bl}(c_s^{-1}\alpha_s \, c_s), & \mbox{if}& c_s\neq 0,\\
0,& \mbox{if}& c_s=0.\end{array}\right.
\label{5.18}
\end{equation}  
Then
\begin{equation}
\widetilde{f}_{{\boldsymbol \ell},s}(z)=P_{{\Lambda}_s,{\boldsymbol\ell}}(z)\cdot \rho_s\cdot 
P_{\Omega,{\bf r}}(z)
\label{5.19}  
\end{equation} 
is a unique polynomial of degree less than $n+m$ satisfying conditions \eqref{5.17} and 
\begin{equation}
f^{\bl}(\alpha_s)=c_s,\quad f^{\bl}(\alpha_i)=0 \; \; (i\neq s),\quad f^{\br}(\beta_j)=d_j\; \; (j=1,\ldots,m).
\label{5.20}
\end{equation}
\label{L:5.5}
\end{lemma}
\begin{proof} The proof is similar to that of Lemma \ref{L:5.4}. If $c_s=0$, the statement follows from Theorem 
\ref{T:5.1}. If $c_s\neq 0$, the formula \eqref{5.18} makes sense, since 
$P_{{\Lambda}_s,{\boldsymbol\ell}}(\alpha_s)\neq 0$ 
and since  $\alpha_s\not\in\bigcup_{j=1}^m[\beta_j]=\cZ(P_{\Omega,{\bf r}}^{\sharp}P_{\Omega,{\bf r}})$. 
If $c_s\neq 0$, it is seen from \eqref{5.19} that $\deg(\widetilde{f}_{{\boldsymbol \ell},s})=
\deg(P_{\Lambda_s,{\boldsymbol\ell}})+\deg(P_{{\Omega},{\bf r}})=n-1+m$, while the uniqueness of a low-degree solution
follows from Theorem \ref{5.1} (as in the proof of Lemma \ref{L:5.4}). It remains to show that 
$\widetilde{f}_{{\boldsymbol \ell},s}$ indeed satisfies conditions \eqref{5.17} and \eqref{5.20}.

\smallskip

Since $P_{\Omega,{\bf r}}^{\br}(\beta_j)=0$ for $1\le\le m$, we have for the right backward shift 
$R_{\beta_j}$,
\eqref{4.1a},
$$
R_{\beta_j}\widetilde{f}_{{\boldsymbol \ell},s}=
R_{\beta_j}\left(P_{{\Lambda}_s,{\boldsymbol\ell}}\cdot \rho_s\cdot
P_{\Omega,{\bf r}}\right)=P_{{\Lambda}_s,{\boldsymbol\ell}}\cdot \rho_s\cdot
(R_{\beta_j}P_{\Omega,{\bf r}}).
$$
Since $P_{{\Lambda}_s,{\boldsymbol\ell}}(\alpha_i)=0$ for $i=1,\ldots,k$ (recall that $s>k$), we have in particular,
$(R_{\beta_i}\widetilde{f}_{{\boldsymbol \ell},s})^{\bl}(\alpha_i)=0$ for $i=1,\ldots,k$, and the latter equalities are 
equivalent to \eqref{5.17} due to \eqref{5.3}. 

\smallskip

Furthermore, if $c_s\neq 0$,  the formula \eqref{5.18} can be written 
equivalently as 
$$
P_{{\Lambda}_s,{\boldsymbol \ell}}^{\bl}(\alpha_s)\cdot \rho_s=
\left[(P_{\Omega,{\bf r}}^{\sharp}P_{\Omega,{\bf r}})(\alpha_s)\right]^{-1}\cdot c_s\cdot
P_{\Omega,{\bf r}}^{\sharp\bl}(c_s^{-1}\alpha_s \, c_s)
$$
and then by implication $\Leftarrow$ in \eqref{5.13} (with $\delta=P_{{\Lambda}_s,{\boldsymbol \ell}}^{\bl}(\alpha_s)\cdot 
\rho_s$), we have
$$
c_s=P_{{\Lambda}_s,{\boldsymbol \ell}}^{\bl}(\alpha_s)\cdot \rho_s\cdot
P_{\Omega,{\bf r}}^{\bl}\left(\rho_s^{-1}\cdot P_{{\Lambda}_s,{\boldsymbol \ell}}^{\bl}(\alpha_s)^{-1}
\cdot \alpha_s\cdot  P_{{\Lambda}_s,{\boldsymbol \ell}}^{\bl}(\alpha_s)\cdot \rho_s\right).
$$
On the other hand, formula \eqref{2.6} applied to the product \eqref{5.19} gives 
$$
\widetilde{f}_{{\boldsymbol \ell},s}^{\bl}(\alpha_s)=P_{{\Lambda}_s,{\boldsymbol \ell}}^{\bl}(\alpha_s)\cdot 
\rho_s\cdot
P_{\Omega,{\bf r}}^{\bl}\left(\rho_s^{-1}\cdot P_{{\Lambda}_s,{\boldsymbol \ell}}^{\bl}(\alpha_s)^{-1}
\cdot \alpha_s \cdot P_{{\Lambda}_s,{\boldsymbol \ell}}^{\bl}(\alpha_s)\cdot \rho_s\right),
$$
and the two latter equalities imply $\widetilde{f}_{{\boldsymbol \ell},i}^{\bl}(\alpha_s)=c_s$. Verification of all 
other equalities in \eqref{5.20} is the same as in the proof of Lemma \ref{L:5.4}. 
\end{proof}
\begin{lemma}
Under assumptions \eqref{5.1} and \eqref{5.2}, let $s\in\{1,\ldots,k\}$, let 
$\widetilde{\alpha}_s$ and $\widetilde{\beta}_s$ be defined as in \eqref{5.9} (so that 
$\alpha_s,\beta_s,\widetilde{\alpha}_s,\widetilde{\beta}_s$ belong to the same conjugacy class)
and let $\Pi_{\widetilde{\alpha}_s,\widetilde{\beta}_s}$ be the plane defined via formula 
\eqref{4.9}.
Then 
\begin{enumerate}
\item Quaternions  $\gamma_s$ and $\rho_s$ given by (compare with \eqref{5.14} and \eqref{5.18})
\begin{equation}
\gamma_s=\left\{\begin{array}{ccc} P_{\Lambda_s,{\boldsymbol\ell}}^{\sharp\br}(d_s\beta_sd_s^{-1})\cdot d_s\cdot
(P_{{\Lambda}_s,{\boldsymbol\ell}}^{\sharp}P_{{\Lambda}_s,{\boldsymbol\ell}})(\beta_s)^{-1}\cdot
P_{{\Omega}_s,{\bf r}}^{\br}(\beta_s)^{-1},&\mbox{if}& d_s\neq 0,\\ 0,&\mbox{if}& d_s=0,\end{array}\right.
\label{5.21}
\end{equation}  
\begin{equation}
\rho_s=\left\{\begin{array}{ccc}P_{{\Lambda}_s,{\boldsymbol \ell}}^{\bl}(\alpha_s)^{-1}
\cdot(P_{{\Omega}_s,{\bf r}}^{\sharp}P_{{\Omega}_s,{\bf r}})(\alpha_s)^{-1}\cdot c_s\cdot
P_{\Omega_s,{\bf r}}^{\sharp\bl}(c_s^{-1}\alpha_s c_s), & \mbox{if}& c_s\neq 0,\\
0,& \mbox{if}& c_s=0.\end{array}\right.
\label{5.22}
\end{equation}
are subject to equality 
\begin{equation}
\overline{\widetilde{\alpha}_s}(\rho_s-\gamma_s)=(\rho_s-\gamma_s)\widetilde{\beta}_s.
\label{5.23}
\end{equation}
\item All polynomials $f$ of degree less than $n+m$ satisfying conditions 
\begin{align}
&f^{\bl}(\alpha_s)=c_s,\quad f^{\bl}(\alpha_j)=0 \quad (j\in\{1,\ldots,n\}\backslash\{s\}),\label{5.24}\\
&f^{\br}(\beta_s)=d_s,\quad f^{\br}(\beta_j)=0 \quad (j\in\{1,\ldots,m\}\backslash\{s\})\label{5.25},\\
&(L_{\alpha_i} f)^{\br}(\beta_i)=0  \quad (i\in\{1,\ldots,k\}\backslash\{s\}),\label{5.26}
\end{align}
are given by the formula
\begin{equation}
f(z)=\widetilde{f}_{s}(z)+P_{\Lambda,{\boldsymbol \ell}}(z)\cdot \mu_s\cdot P_{{\Omega}_s,{\bf r}}(z)
\label{5.27}
\end{equation}
where $\mu_s\in \Pi_{\widetilde{\alpha}_s,\widetilde{\beta}_s}$ is a free parameter and where 
\begin{equation}
\widetilde{f}_{s}(z)=P_{\Lambda_s,{\boldsymbol \ell}}(z)\cdot \rho_s \cdot P_{{\Omega}_s,{\bf r}}(z)+
P_{\Lambda,{\boldsymbol \ell}}(z)\cdot (2{\rm
Im}(\widetilde{\alpha}_s))^{-1}(\rho_s-\gamma_s)\cdot P_{{\Omega}_s,{\bf r}}(z).\label{5.28}
\end{equation}
\end{enumerate}
\label{L:5.6}
\end{lemma}
\begin{proof} We first observe that $d_s$ and $c_s$ are recovered from \eqref{5.21} and \eqref{5.22} by
\begin{equation}
d_s=\left(P_{{\Lambda}_s,{\boldsymbol \ell}}\cdot\gamma_s\cdot P_{{\Omega}_s,{\bf r}}\right)^{\br}(\beta_s),
\qquad c_s=\left(P_{{\Lambda}_s,{\boldsymbol \ell}}\cdot\rho_s\cdot P_{{\Omega}_s,{\bf 
r}}\right)^{\bl}(\alpha_s).
\label{5.29}
\end{equation}
The trivial cases where $d_s=c_s=0$ are clear. If $d_s\neq 0$, we have from \eqref{5.21},
$$
\gamma_s\cdot P_{{\Omega}_s,{\bf r}}^{\br}(\beta_s)=
P_{\Lambda_s,{\boldsymbol\ell}}^{\sharp\br}(d_s\beta_sd_s^{-1})\cdot 
d_s\cdot
(P_{{\Lambda}_s,{\boldsymbol\ell}}^{\sharp}P_{{\Lambda}_s,{\boldsymbol\ell}})(\beta_s)^{-1}
$$
and by implication $\; \Leftarrow\; $ in \eqref{5.12} and formula \eqref{2.7} we conclude
\begin{align*}
d_s&=P_{{\Lambda}_s,{\boldsymbol \ell}}^{\br}\left(\gamma_s\cdot P_{{\Omega}_s,{\bf r}}^{\br}(\beta_s)
\cdot\beta_s\cdot P_{{\Omega}_s,{\bf r}}^{\br}(\beta_s)^{-1}\cdot\gamma_s^{-1}
\right)\cdot\gamma_s\cdot P_{{\Omega}_s,{\bf r}}^{\br}(\beta_s)\\
&=\left(P_{{\Lambda}_s,{\boldsymbol \ell}}\cdot\gamma_s\cdot P_{{\Omega}_s,{\bf r}}\right)^{\br}(\beta_s),
\end{align*}
which confirms the first equality in \eqref{5.29}. The second equality for $c_s\neq 0$ is 
verified 
in much the same way. Let us introduce 
\begin{equation}
\widetilde{d}_s:=\left(P_{{\Lambda}_s,{\boldsymbol \ell}}\cdot\rho_s\cdot P_{{\Omega}_s,{\bf 
r}}\right)^{\br}(\beta_s).
\label{5.30}
\end{equation}
It follows from \eqref{5.29} and \eqref{5.30} that $c_s$ and $\widetilde{d}_s$ are left and right 
values 
of the same polynomial at conjugate points $\alpha_s$ and $\beta_s$; therefore 
$$
\overline{\alpha}_s(c_s-\widetilde{d}_s)=(c_s-\widetilde{d}_s)\beta_s,
$$
by Corollary \ref{C:4.4}. On the other hand, by the assumption \eqref{5.2},
\begin{equation}
\overline{\alpha}_s(c_s-d_s)=(c_s-d_s)\beta_s.
\label{5.30a}
\end{equation}
Combining the two latter equalities gives 
$\overline{\alpha}_s(\widetilde{d}_s-d_s)=(\widetilde{d}_s-d_s)\beta_s$. Substituting the formulas 
\eqref{5.29} and \eqref{5.30} for $d_s$ and $\widetilde{d}_s$ into the latter equality and taking 
into 
account that right evaluation is linear on $\bH[z]$, we get 
$$
\overline{\alpha}_s\cdot \left(P_{{\Lambda}_s,{\boldsymbol \ell}}\cdot(\rho_s-\gamma_s)\cdot 
P_{{\Omega}_s,{\bf r}}\right)^{\br}(\beta_s)=\left(P_{{\Lambda}_s,{\boldsymbol 
\ell}}\cdot(\rho_s-\gamma_s)\cdot
P_{{\Omega}_s,{\bf r}}\right)^{\br}(\beta_s)\cdot\beta_s
$$
which can be equivalently written as 
$$
\left((z-\overline{\alpha}_s)\cdot P_{{\Lambda}_s,{\boldsymbol \ell}}\cdot(\rho_s-\gamma_s)\cdot
P_{{\Omega}_s,{\bf r}}\right)^{\br}(\beta_s)=0.
$$
Using formula \eqref{2.7} and definition \eqref{5.9} of $\widetilde{\beta}_s$, we write 
the latter formula as 
$$
\left((z-\overline{\alpha}_s)\cdot P_{{\Lambda}_s,{\boldsymbol \ell}}\cdot(\rho_s-\gamma_s)
\right)^{\br}(\widetilde{\beta}_s)\cdot P_{{\Omega}_s,{\bf r}}^{\br}(\beta_s)=0
$$
and since $P_{{\Omega}_s,{\bf r}}^{\br}(\beta_s)\neq 0$, we have 
\begin{equation}
\left((z-\overline{\alpha}_s)\cdot P_{{\Lambda}_s,{\boldsymbol \ell}}\cdot(\rho_s-\gamma_s)     
\right)^{\br}(\widetilde{\beta}_s)=0.
\label{5.31} 
\end{equation}
Observe that if $\Lambda=\{\alpha_1,\ldots,\alpha_n\}$ is permuted by
moving $\alpha_s$ to the rightmost spot, then the 
recursion \eqref{2.14} produces  $P_{\Lambda,{\boldsymbol\ell}}$ in the form
\begin{equation}
P_{\Lambda,{\boldsymbol\ell}}(z)=P_{{\Lambda}_s,{\boldsymbol\ell}}(z)\cdot
\left(z-P^{\bl}_{{\Lambda}_s,{\boldsymbol\ell}}(\alpha_s)^{-1}\alpha_s
P^{\bl}_{{\Lambda}_s,{\boldsymbol\ell}}(\alpha_s) \right)=
P_{{\Lambda}_s,{\boldsymbol\ell}}(z)\cdot
(z-\widetilde{\alpha}_s).
\label{5.32}
\end{equation}
Since $P_{\Lambda,{\boldsymbol\ell}}$ has only one left zero in the conjugacy class $[\alpha_s]$, we have
\begin{align*}
(z-\overline{\alpha}_s)\cdot P_{\Lambda,{\boldsymbol\ell}}(z)&=
\cX_{[\alpha_s]}(z)\cdot (L_{\alpha_s}P_{\Lambda,{\boldsymbol\ell}})(z)\\
&=(L_{\alpha_s}P_{\Lambda,{\boldsymbol\ell}})(z)\cdot \cX_{[\alpha_s]}(z)\\
&=(L_{\alpha_s}P_{\Lambda,{\boldsymbol\ell}})(z)\cdot \cX_{[\widetilde{\alpha}_s]}(z)=
(L_{\alpha_s}P_{\Lambda,{\boldsymbol\ell}})(z)\cdot 
(z-\overline{\widetilde{\alpha}}_s)(z-\widetilde{\alpha}_s)
\end{align*}
which being compared with \eqref{5.32} implies
\begin{equation}
(L_{\alpha_s}P_{\Lambda,{\boldsymbol\ell}})(z)\cdot
(z-\overline{\widetilde{\alpha}}_s)=(z-\overline{\alpha}_s)\cdot P_{{\Lambda}_s,{\boldsymbol \ell}}(z).
\label{5.32a}
\end{equation}
Substituting this identity into \eqref{5.31} gives
$$
\left((L_{\alpha_s}P_{\Lambda,{\boldsymbol\ell}})\cdot (z-\overline{\widetilde{\alpha}}_s)
\cdot(\rho_s-\gamma_s)\right)^{\br}(\widetilde{\beta}_s)=0.
$$
Since the polynomial $L_{\alpha_s}P_{\Lambda,{\boldsymbol\ell}}$ has no zeros in the conjugacy class 
$[\alpha_s]=[\widetilde{\beta}_s]$, the latter equality implies 
$$
\left((z-\overline{\widetilde{\alpha}}_s)
\cdot(\rho_s-\gamma_s)\right)^{\br}(\widetilde{\beta}_s)=0
$$
which is the same as \eqref{5.23}. This completes the proof of the first statement of the lemma.

\smallskip

To prove the second statement, we first exclude the first condition in \eqref{5.25} and 
consider the problem with the remaining interpolation conditions in \eqref{5.24}--\eqref{5.26}.
In this reduced setting, the role of $\Omega$ is played by the set ${\Omega}_s$
(the point $\beta_s$ is temporarily excluded from $\Omega$)  and, since $\alpha_1,\ldots,\alpha_n$ does not 
have conjugates in $\Omega_s$, the reduced problem is of the type considered in Lemma \ref{L:5.5}. Combining 
Lemma \ref{L:5.5} and Theorem \ref{T:5.1} we conclude that all solutions of the reduced problem are given 
by the formula 
\begin{equation}
f(z)=P_{{\Lambda}_s,{\boldsymbol\ell}}(z)\cdot \rho_s\cdot
P_{{\Omega}_s,{\bf r}}(z)+P_{\Lambda,{\boldsymbol \ell}}(z)\cdot g(z)\cdot P_{{\Omega}_s,{\bf 
r}}(z)
\label{5.33}
\end{equation}
for some $g\in\bH[z]$ (by Theorem \ref{T:5.1}, the second term on the right side of \eqref{5.33} 
is a 
general  solution to the reduced homogeneous problem). Note that $\rho_s$ in \eqref{5.33} is 
defined by 
formula \eqref{5.18} but with $\Omega$ replaced by ${\Omega}_s$, that is, by formula \eqref{5.22}.
Let us observe the factorization 
\begin{equation}
P_{{\Omega},{\bf r}}(z)=(z-\widetilde{\beta}_s)\cdot P_{{\Omega}_s,{\bf r}}(z)
\label{5.34}
\end{equation}!
which follows from the recursion \eqref{2.13} applied to the set $\Omega=\{\beta_1,\ldots,\beta_m\}$
permuted by moving $\beta_s$ to the rightmost spot (the right-sided counterpart of \eqref{5.32}).
Furthermore, let us take $g$ in the form $g(z)=\psi_s+h(z)\cdot (z-\widetilde{\beta}_s)$ where 
$\psi_s=g^{\br}(\widetilde{\beta}_s)$ and $h=R_{\widetilde{\beta}_s}g$. Substituting this representation
into \eqref{5.33} and taking into account \eqref{5.34} we get
$$
f(z)=P_{{\Lambda}_s,{\boldsymbol\ell}}(z)\cdot \rho_s\cdot
P_{{\Omega}_s,{\bf r}}(z)+P_{\Lambda,{\boldsymbol \ell}}(z)\cdot \psi_s\cdot P_{{\Omega}_s,{\bf
r}}(z)+P_{{\Lambda}_s,{\boldsymbol\ell}}(z)\cdot h(z)\cdot P_{{\Omega},{\bf r}}(z).
$$
The rightmost term on the right represents the general solution to the homogemneous problem \eqref{5.3}, 
\eqref{5.6} and is of degree at least $m+n$ if $h\not\equiv 0$. Let us focus on the low-degree part
corresponding to the choice of $h\equiv 0$. The main objective now is to specify $\psi_s$ in such a way
that the function 
\begin{equation}
f=P_{{\Lambda}_s,{\boldsymbol\ell}}\cdot \rho_s\cdot
P_{{\Omega}_s,{\bf r}}+P_{\Lambda,{\boldsymbol \ell}}\cdot \psi_s\cdot P_{{\Omega}_s,{\bf
r}}
\label{5.35}  
\end{equation}
will satisfy  the first condition in \eqref{5.25}. Making use of factorizations \eqref{5.32}, 
\eqref{5.34} and setting
\begin{equation}
\Phi_s:=\rho_s+\psi_s\widetilde{\beta}_s-\widetilde{\alpha}_s\psi_s,
\label{5.36}
\end{equation}
we rewrite \eqref{5.35}
\begin{align}
f(z)=&P_{{\Lambda}_s,{\boldsymbol\ell}}(z)\cdot \rho_s\cdot
P_{{\Omega}_s,{\bf r}}(z)+P_{{\Lambda}_s,{\boldsymbol \ell}}(z)\cdot(z-\widetilde{\alpha}_s)
\cdot \psi_s\cdot P_{{\Omega}_s,{\bf r}}(z)\notag\\
=&P_{{\Lambda}_s,{\boldsymbol\ell}}(z)\cdot \Phi_s\cdot
P_{{\Omega}_s,{\bf r}}(z)+ P_{{\Lambda}_s,{\boldsymbol \ell}}(z)\cdot \psi_s\cdot 
(z-\widetilde{\beta}_s)\cdot P_{{\Omega}_s,{\bf r}}(z)\notag\\
=&P_{{\Lambda}_s,{\boldsymbol\ell}}(z)\cdot \Phi_s\cdot
P_{{\Omega}_s,{\bf r}}(z)+ P_{{\Lambda}_s,{\boldsymbol \ell}}(z)\cdot \psi_s\cdot
P_{{\Omega},{\bf r}}(z).\label{5.37}
\end{align}
Since $P_{{\Omega},{\bf r}}^{\br}(\beta_s)=0$, we conclude that $f$ of the form \eqref{5.37} satisfies 
condition  $f^{\br}(\beta_s)=d_s$ if and only if 
\begin{equation}
\left(P_{{\Lambda}_s,{\boldsymbol\ell}}\cdot \Phi_s\cdot
P_{{\Omega}_s,{\bf r}}\right)^{\br}(\beta_s)=d_s.
\label{5.38}
\end{equation}
Since the conjugacy class $[\beta_s]=[\alpha_s]$ is disjoint with 
$\cZ_{\bf r}(P_{{\Lambda}_s,{\boldsymbol\ell}})$ and  $\cZ_{\bf r}(P_{{\Omega}_s,{\bf r}})$, 
it follows from \eqref{5.38} and \eqref{5.21} that $\Phi_s=0 \Leftrightarrow d_s=0 
\Leftrightarrow \gamma_s=0$.
If $d_s\neq 0$, then $\Phi_s$ is uniquely recovered from \eqref{5.38} as well as $\gamma_s$ is 
recovered from 
the first formula in \eqref{5.29}. We then conclude from \eqref{5.29} and \eqref{5.38} that 
$\Phi_s=\gamma_s$.
We summarize: $f$  satisfies conditions \eqref{5.24}--\eqref{5.26} and is of degree less than $n+m$  
if and only if it is of the form \eqref{5.35} where $\psi_s$ satisfies \eqref{5.36}, i.e. (since 
$\Phi_s=\gamma_s$), if and only if $\psi_s$ is a solution to the Sylvester equation  
\begin{equation}
\widetilde{\alpha}_s\psi_s-\psi_s\widetilde{\beta}_s=\rho_s-\gamma_s.
\label{5.39}
\end{equation}
The latter equation is consistent, by Lemma \ref{L:6.2} and due to equality \eqref{5.23}.
By Lemma \ref{L:6.2}, all solutions $\psi_s$ to the equation \eqref{5.39} are given by the 
formula 
$$
\psi_s=(2{\rm Im}(\widetilde{\alpha}_s))^{-1}(\rho_s-\gamma_s)+\mu_s,\qquad \mu_s\in
\Pi_{\widetilde{\alpha}_s,\widetilde{\beta}_s}
$$
which being substituted into \eqref{5.35}, gives \eqref{5.26} completing the proof.
\end{proof}
\begin{remark}
{\rm The second term on the right side of \eqref{5.27} looks asymmetric with respect to 
the sets $\Lambda$ and $\Omega$. However, since the membership
$\mu_s\in\Pi_{\widetilde{\alpha}_s,\widetilde{\beta}_s}$ means that
$\widetilde{\alpha}_s\mu_s=\mu_s\widetilde{\beta}_s$ (by Lemma \ref{L:6.2}), 
we have 
$\; 
(z-\widetilde{\alpha}_s)\cdot \mu_s=\mu_s\cdot (z-\widetilde{\beta}_s), \;$
and the alternative representation 
$$
P_{\Lambda,{\boldsymbol \ell}}\cdot \mu_s\cdot P_{{\Omega}_s,{\bf r}}
=P_{\Lambda_s,{\boldsymbol \ell}}\cdot(z-\widetilde{\alpha}_s)\cdot \mu_s \cdot P_{{\Omega}_s,{\bf r}}
=P_{\Lambda_s,{\boldsymbol \ell}}\cdot \mu_s \cdot(z-\widetilde{\beta}_s)\cdot P_{{\Omega}_s,{\bf r}}=
P_{\Lambda_s,{\boldsymbol \ell}}\cdot \mu_s \cdot P_{{\Omega},{\bf r}}
$$
follows from \eqref{5.32} and \eqref{5.36}}.
\label{R:5.6a}
\end{remark}
The reason for non-uniqueness of a low-degree solution to the problem 
\eqref{5.24}--\eqref{5.26} is that the value of $(L_{\alpha_s}f)^{\br}(\beta_s)$ is not fixed.
\begin{remark}
For each $q_s\in\bH$ satisfynig the Sylvester equality
\begin{equation}
\alpha_sq-q\beta_s=c_s-d_s,
\label{5.40a}
\end{equation}
there exists a unique polynomial $f$ of the form \eqref{5.27} such that  $(L_{\alpha_s}f)^{\br}(\beta_s)=q_s$.
\label{R:5.7} 
\end{remark}
\begin{proof}
For any  $\mu_s\in \Pi_{\widetilde{\alpha}_s,\widetilde{\beta}_s}$ and  $f$ of the form \eqref{5.27}, we have
\begin{equation}
(L_{\alpha_s}f)^{\br}(\beta_s)=(L_{\alpha_s}\widetilde{f}_s)^{\br}(\beta_s)+   
((L_{\alpha_s}P_{\Lambda,{\boldsymbol \ell}})\cdot\mu_s\cdot P_{{\Omega}_s,{\bf
r}})^{\br}(\beta_s),
\label{5.40b}
\end{equation}
and the quaternion $q_s:=(L_{\alpha_s}f)^{\br}(\beta_s)$ satisfies equality \eqref{5.40a},
by Remark \ref{R:4.4a}. On the other hand, for any $\widetilde{q}_s\in\bH$, the equation 
$$
((L_{\alpha_s}P_{\Lambda,{\boldsymbol \ell}})\cdot\mu_s\cdot P_{{\Omega}_s,{\bf
r}})^{\br}(\beta_s)=\widetilde{q}_s
$$
can be solved for $\mu_s$ (using the implication $\; \Rightarrow\;$ in \eqref{5.12}) as follows:
\begin{equation}
\mu_s=(L_{\alpha_s}P_{\Lambda,{\boldsymbol\ell}})^{\sharp\br}(\widetilde{q}_s\beta_s 
\widetilde{q}_s^{-1})\cdot d_s\cdot
(P_{\Lambda_s,{\boldsymbol\ell}}^{\sharp}P_{\Lambda,{\boldsymbol\ell}_s})(\beta_s)^{-1}\cdot
P_{{\Omega}_s,{\bf r}}^{\br}(\beta_s)^{-1}, \; \; \mbox{if} \; \; \widetilde{q}_s\neq 0,
\label{5.41}
\end{equation}
and $\mu_s=0$ if $\widetilde{q}_s=0$; in fact, we should have used $L_{\alpha_s}P_{\Lambda,{\boldsymbol\ell}}$
rather than $P_{\Lambda,{\boldsymbol\ell}_s}$ in the latter formula but, although these polynomials are distinct in general, 
it follows from identity \eqref{5.32a} and the third equality in \eqref{2.3} that 
$(L_{\alpha_s}P_{\Lambda,{\boldsymbol\ell}})^\sharp 
(L_{\alpha_s}P_{\Lambda,{\boldsymbol\ell}})=P_{\Lambda,{\boldsymbol\ell}_s}^\sharp P_{\Lambda,{\boldsymbol\ell}_s}$.
We now conclude from \eqref{5.40b} that $f$ of the form \eqref{5.27} satisfies equality 
$(L_{\alpha_s}f)^{\br}(\beta_s)=q_s$ if and only if the parameter $\mu_s$ is defined by formula
\eqref{5.41} with $\widetilde{q}_s=q_s-(L_{\alpha_s}\widetilde{f}_s)^{\br}(\beta_s)$.
\end{proof}
The next theorem is the main result of the paper.
\begin{theorem}
Under assumptions \eqref{5.1} and \eqref{5.2}, let
$\widetilde{\alpha}_s$ and $\widetilde{\beta}_s$ be defined as in \eqref{5.9}, 
let $\Pi_{\widetilde{\alpha}_s,\widetilde{\beta}_s}$ be the plane defined via formula 
\eqref{4.9} and let polynomials $\widetilde{f}_{{\bf r},j}$ ($k<j\le m$), 
$\widetilde{f}_{{\boldsymbol\ell},i}$ ($k<i\le n$) and $\widetilde{f}_s$ ($1\le s\le k$)
be defined via formulas \eqref{5.15}, \eqref{5.19} and \eqref{5.28}, respectively.
Then all polynomials $f\in\bH[z]$ satisfying conditions \eqref{1.8}, \eqref{1.9}
are given by the formula
\begin{align}
f(z)=&\sum_{i=k+1}^n\widetilde{f}_{{\boldsymbol\ell},i}(z)+
\sum_{j=k+1}^m\widetilde{f}_{{\bf r},j}(z)+\sum_{s=1}^k\widetilde{f}_{s}(z)
+\sum_{s=1}^k P_{\Lambda_s,{\boldsymbol \ell}}(z)\cdot \mu_s\cdot P_{{\Omega},{\bf r}}(z)\notag\\
&+P_{\Lambda,{\boldsymbol \ell}}(z)\cdot h(z)\cdot P_{{\Omega},{\bf r}}(z),
\label{5.40}
\end{align}
where $\mu_s\in \Pi_{\widetilde{\alpha}_s,\widetilde{\beta}_s}$ and $h\in\bH[z]$ are free 
parameters. 
\label{T:5.7}
\end{theorem}
\begin{proof}
The ``if" part follows immediately from Theorem \ref{T:5.1}, and Lemmas \ref{L:5.4}, \ref{L:5.5}, \ref{L:5.6}.
Now let $f$ be any polynomial satisfying conditions \eqref{1.8}, \eqref{1.9}. For each $s\in\{1,\ldots,k\}$,
define $\mu_s$ a unique element $\mu_s\in\bH$ (as explained in Remark \ref{R:5.7}) such that 
$$
\left(L_{\alpha_s}(P_{\Lambda,{\boldsymbol \ell}}\cdot \mu_s\cdot P_{{\Omega}_s,{\bf r}})\right)^{\br}(\beta_s)=
\left(L_{\alpha_s}f\right)^{\br}(\beta_s)-(L_{\alpha_s}\widetilde{f}_{s})^{\br}(\beta_s).
$$
Then the polynomial 
$\; f-{\displaystyle\sum_{i=k+1}^n\widetilde{f}_{{\boldsymbol\ell},i}-
\sum_{j=k+1}^m\widetilde{f}_{{\bf r},j}-\sum_{s=1}^k\widetilde{f}_{s}-
\sum_{s=1}^k P_{\Lambda_s,{\boldsymbol \ell}}\cdot \mu_s\cdot P_{{\Omega},{\bf r}}}\; $
solves the homogeneous problem \eqref{5.3}, \eqref{5.6} and hence, belongs to $P_{\Lambda,{\boldsymbol 
\ell}}\cdot\bH[z]\cdot  P_{{\Omega}_s,{\bf r}}$ from which representation \eqref{5.40} follows.
\end{proof}
Combining Theorem \ref{T:5.7} and Remark \ref{R:5.7} one can get the non-homogeneous version of Theorem \ref{T:5.1}.
\begin{theorem}
Under assumptions \eqref{5.1} and \eqref{5.2}, let $q_s$ be a solution to the Sylvester equation \eqref{5.40a} for 
$s=1,\ldots,k$. The set of all $f\in\bH[z]$ satisfying conditions \eqref{1.8}, \eqref{1.9} and 
$$
(L_{\alpha_s} f)^{\br}(\beta_s)=q_s\quad\mbox{for}\quad s=1,\ldots,k
$$
is parametrized by formula \eqref{5.40} with free parameter $h\in\bH[z]$ and the elements $\mu_s$ uniquely
determined by $q_s$ (as explained in Remark \ref{R:5.7}).
\label{T:5.7a}   
\end{theorem}
{\bf Proof of Theorem \ref{T:5.2}:} We let $c_i=d_j=0$ for all $i=1,\ldots,n$ and $j=1,\ldots,m$ and then conclude from 
formulas \eqref{5.14}, \eqref{5.18}, \eqref{5.21}, \eqref{5.22} that $\rho_i=\gamma_j=0$ for all $i=1,\ldots,n$ and 
$j=1,\ldots,m$. Then it follows from formulas \eqref{5.15}, \eqref{5.19}, \eqref{5.28} that all elementary polynomials 
$\widetilde{f}_{{\boldsymbol\ell},i}$, $\widetilde{f}_{{\bf r},j}$, $\widetilde{f}_{s}$ are zero polynomials.
Now description \eqref{5.10} in Theorem \ref{T:5.2} follows from \eqref{5.40}.\qed

\medskip

Another particular case of Theorem \ref{T:5.7} (where $k=0$) admits a fairly simple answer.
\begin{theorem}
Assume that $[\alpha_i]\cap \Omega=\emptyset \; (1\ge i\le n)$ and 
$[\beta_j]\cap \Lambda=\emptyset \; (1\le j\le m)$. 
Then all polynomials $f\in\bH[z]$ satisfying conditions \eqref{1.8}, \eqref{1.9}
are given by 
\begin{equation}
f(z)=\sum_{i=1}^n\widetilde{f}_{{\boldsymbol\ell},i}(z)+
\sum_{j=1}^m\widetilde{f}_{{\bf r},j}(z)
+P_{\Lambda,{\boldsymbol \ell}}(z)\cdot h(z)\cdot P_{{\Omega},{\bf r}}(z),\quad h\in\bH[z]
\label{5.47}
\end{equation}
where $\widetilde{f}_{{\bf r},j}$ and $\widetilde{f}_{{\boldsymbol\ell},i}$ are 
defined via formulas \eqref{5.15}, \eqref{5.19}. The two first terms on the right 
side of \eqref{5.47} present a unique polynomial of degree less than $m+n$ satisfying conditions 
\eqref{1.8}, \eqref{1.9}.
\label{T:5.9}
\end{theorem}
Specializing formula \eqref{5.47} further to the case where $\Omega=\emptyset$ and therefore,
$P_{\Omega,{\bf r}}\equiv 1$ and
$\rho_i=P^{\bl}_{{\Lambda}_i,{\boldsymbol\ell}}(\alpha_i)^{-1}$ (according to \eqref{5.18}),
recovers Theorem \ref{T:5.1b}.
Letting $\Lambda=\emptyset$ in Theorem \ref{T:5.9} and making appropriate adjustments we recover 
Theorem \ref{T:5.1a}.

\section{Alternative formulas for low-degree particular solutions}
\setcounter{equation}{0}   

As in the complex case, a low-degree solution (in case it is unique) can be constructed 
via several different schemes. Although the produced formulas are not as explicit in 
terms of interpolation data as the Lagrange's formula \eqref{1.2}, the algorithms might be
more efficient from the computational point of view. For the left-sided problem \eqref{1.8},
one can pick any basis $\{{\bf a}_1,\ldots,{\bf a}_n\}$ for the space of polynomials of 
degree less than $n$ and find
$\widetilde{f}_{\boldsymbol\ell}$ in the form 
$\widetilde{f}_{\boldsymbol\ell}(z)=\sum {\bf a}_j(z)\varphi_j$ with the
coefficients $\varphi_i$ obtained
from the linear system
\begin{equation}
\sum_{j=1}^n {\bf a}^{\bl}_j(\alpha_i)\varphi_j=c_i\quad\mbox{for}\quad i=1,\ldots,n. 
\label{6.1}
\end{equation}
Theorem \ref{T:5.1} implies in particular, that the latter system has a unique  
solution for any choice of $c_1,\ldots,c_n$. This, in turn, implies that the matrix 
$A$ of the system \eqref{6.1} is 
invertible and therefore, the left Lagrange polynomial can be written as 
\begin{equation}
\widetilde{f}_{\boldsymbol\ell}(z)=\begin{bmatrix}{\bf a}_1(z)
& \ldots & {\bf a}_n(z)\end{bmatrix}A^{-1}C,\quad
A=\left[{\bf a}^{\bl}_j(\alpha_i)\right]_{i,j=1}^n,\quad C=\begin{bmatrix}c_1 \\
\vdots \\ c_{n}\end{bmatrix}.
\label{6.2}
\end{equation}
We mention three ``canonical" bases. If we let ${\bf a}_j=P_{\Lambda_j,{\boldsymbol\ell}}$ 
to be the {\bf lmp}
of the set ${\Lambda}_j:=\Lambda\setminus\{\alpha_j\}$, the matrix $A$ is  
diagonal and the formula \eqref{6.2} amounts to the Lagrange formula
\eqref{3.8}. The monomial basis ${\bf a}_j(z)=z^{j-1}$ ($j=1,\ldots,n$) leads to the 
Vandermonde matrix $A=\left[\alpha_i^{j-1}\right]_{i,j=1}^n$. 
The third convenient basis ${\bf a}_j=p_{j-1}$ ($j=1,\ldots,n$)
is suggested by the recursion \eqref{2.14}.
Recall that $p_0\equiv 1$ and $p_j(z)$ is the {\bf lmp} for the set 
$\{\alpha_1,\ldots,\alpha_j\}$. Therefore, the matrix $A$ is low triangular:
$$
A=\left[p^{\bl}_{k-1}(\alpha_i)\right]_{i,k=1}^n=
\begin{bmatrix}1 & 0 & 0 &\ldots & 0 \\
1 & p^{\bl}_1(\alpha_2) & 0 & \ldots & 0\\
1& p^{\bl}_1(\alpha_3) & p^{\bl}_2(\alpha_3) & \ldots & 0
\\ \vdots &\vdots & \vdots &\ddots & \vdots \\
1& p^{\bl}_1(\alpha_n)& p^{\bl}_2(\alpha_n) & \ldots & p^{\bl}_{n-1}(\alpha_n)
\end{bmatrix},
$$
and the formula \eqref{6.2} now amounts to the Newton's interpolation formula
\begin{equation}
\widetilde{f}_{\boldsymbol\ell}(z)=\varphi_0+p_1(z)\varphi_1+\ldots 
+p_{n-1}(z)\varphi_{n-1},
\label{6.3}
\end{equation}
where the coefficients $\varphi_j$ are recursively recovered from the triangular system
\begin{equation}
\sum_{k=0}^j p_k^{\bl}(\alpha_{j+1})\varphi_{k}=c_j\qquad (j=0,\ldots,n-1). 
\label{6.4}
\end{equation}
Due to the triangular structure, the Newton's scheme easily incorporates additional
interpolation nodes: to get the formula for the Lagrange polynomial satisfying
the additional condition $f(\alpha_{n+1})=c_{n+1}$ (assuming that $\alpha_{n+1}$ is 
equivalent to at most one point from $\{\alpha_1,\ldots,\alpha_n\}$), it suffices to 
calculate  $p_n(z)$, to find $\varphi_{n}$ from \eqref{6.4} and then to add the  extra 
term $p_n(z)\varphi_{n}$
to the expression on the right side of \eqref{6.3}. To apply the Lagrange formula 
\eqref{3.8} in a similar situation, one need to recalculate all basis polynomials
$P_{{\Lambda}_i,{\boldsymbol\ell}}$ for $i=1,\ldots,n+1$.

\smallskip

If the basis polynomials are such that $\deg{\bf a}_j=j-1$ for $j=1,\ldots,n$, then 
the degree of the Lagrange polynomial $\widetilde{f}_{\boldsymbol\ell}$ can be determined 
from interpolation data as the minimal integer $n_0$ such that the column $C$ 
of target values \eqref{6.2} belongs to the right linear span of $n_0$ leftmost columns 
of the  matrix $A$. For example, for the monomial basis ${\bf a}_j=p_{j-1}$, we have
\begin{equation}
\deg(\widetilde{f}_{\boldsymbol\ell})=\min\left\{k: \begin{bmatrix}c_1 \\
\vdots \\ c_{n}\end{bmatrix}\in{\rm span}_{\bf r}\left\{
\begin{bmatrix}\alpha_1^{j-1}\\ \vdots \\ \alpha_n^{j-1}\end{bmatrix}, \; j=1,\ldots,k
\right\}\right\}.
\label{6.4a}
\end{equation}
Similar observations hold true for the right-sided sided problem \eqref{1.9} and can be 
applied to the two-sided problem as follows. Let us assume for the sake of simplicity 
that the low-degree solution is unique, i.e., that is none of the left nodes is 
equivalent to a right node. By Theorem \ref{T:5.1b}, the low degree solution 
must be of the form 
\begin{equation}
\widetilde{f}=\widetilde{f}_{\boldsymbol\ell}+P_{\Lambda,\boldsymbol\ell}h
\label{6.5}
\end{equation}
for some  $h\in\bH[z]$ of degree less than $m$. Formula \eqref{6.5} guarantees that 
$\widetilde{f}$ satisfies left conditions \eqref{1.8}. Let us introduce the elements
$$
\widetilde{d}_j:=d_j-\widetilde{f}_{\boldsymbol \ell}^{\br}(\beta_j)\quad\mbox{for}\quad
j=1,\ldots,m.
$$
Using the equivalence \eqref{5.12}, it can be shown that $\widetilde{f}$ of the 
form \eqref{6.5} satisfies conditions \eqref{1.9} if and only if the parameter $h$ is 
subject to right conditions 
\begin{equation}
h^{\br}(\beta_j)=\delta_j:=\left\{\begin{array}{ccc}
P_{\Lambda,{\boldsymbol\ell}}^{\sharp\br}(\widetilde{d}_j\beta_j\widetilde{d}_j^{-1})\cdot
\widetilde{d}_j\cdot(P_{\Lambda,{\boldsymbol\ell}}^\sharp
P_{\Lambda,{\boldsymbol\ell}})(\beta_j)^{-1},
&\mbox{if}& \widetilde{d}_j\neq 0,\\
0&\mbox{if}& \widetilde{d}_j=0,\end{array}\right.
\label{6.6}
\end{equation}
for $j=1,\ldots,m$. The unique $h$ with $\deg(h)<m$ and satisfying 
conditions \eqref{6.6} can be 
written as
\begin{equation}
h(z)=DB^{-1}\begin{bmatrix}{\bf b}_1(z) \\
\vdots \\ {\bf b}_{m}(z)\end{bmatrix},\quad 
B=\left[{\bf b}^{\br}_i(\beta_j)\right]_{i,j=1}^n,\quad D=\begin{bmatrix}\delta_1 
&\ldots&  \delta_{m}\end{bmatrix}
\label{6.7}
\end{equation}
for a fixed basis $\{{\bf b}_1,\ldots,{\bf b}_m\}$ of the space of polynomials of
degree less than $m$. To construct $h$, we may use the monomial basis and the Vandermonde 
matrix $B=\left[\beta_j^{i-1}\right]_{i,j=1}^m$. Alternatively, we 
may use the right version of the Newton's scheme with the upper triangular matrix 
$B=\left[q^{\br}_{i-1}(\beta_k)\right]_{i,k=1}^n$ where the polynomials $q_j$ 
are constructed recursively in \eqref{2.13}. Similarly to \eqref{6.4a}, we have
\begin{equation}
\deg(h)=\min\left\{s: \begin{bmatrix}\delta_1 &
\cdots& \delta_{m}\end{bmatrix}\in{\rm span}_{\boldsymbol \ell}\left\{
\begin{bmatrix}\beta_1^{i-1}& \cdots& \beta_m^{i-1}\end{bmatrix}, \; i=1,\ldots,s
\right\}\right\}.
\label{6.8}
\end{equation}
Substituting \eqref{6.7} into \eqref{6.5}, we get the unique low degree solution 
$\widetilde{f}$ to the problem \eqref{1.8}, \eqref{1.9}. As we have shown,
$\widetilde{f}$ can be constructed by applying the Newton's scheme first to the left 
problem and then to the recalculated right problem. It would be interesting to 
construct a ``direct" triangular algorithm avoiding the recalculating step.

\smallskip

The value of $\deg(\widetilde{f})$ in terms of interpolation data can be derived from 
\eqref{6.4a}, \eqref{6.5} and \eqref{6.8}. Indeed, 
$\deg(\widetilde{f})=\deg(\widetilde{f}_{\boldsymbol\ell})$ if $\widetilde{d}_j=0$ for 
all
$j=1,\ldots,m$ (i.e., when it happens that $\widetilde{f}_{\boldsymbol\ell}$ accidently
satisfies all right conditions \eqref{1.9}). Otherwise, 
$\deg(\widetilde{f})=\deg (P_{\Lambda,\boldsymbol\ell})+\deg(h)=n+\deg(h)$ where 
the integer $\deg(h)$ is given in \eqref{6.8}. Again, it would be interesting to express 
$\deg(\widetilde{f})$ in terms of the original data rather than the elements 
$\delta_j$'s. A more interesting question is to find minimal degree solutions (or at 
least the value of this minimal degree) in the general setting of Theorem \ref{T:5.9}.
At the moment, it is not clear how small the degree of the polynomial 
in the top line of \eqref{5.40} can be done by varying $\mu_1,\ldots,\mu_k$ in the 
corresponding real planes of $\bH$.

\bibliographystyle{amsplain}

\end{document}